\documentclass[10pt]{amsart}
\usepackage{amsmath,amssymb,amsthm,array,xypic,graphicx,color,enumerate,tikz}

\newcommand\<{\langle}

\newcommand\kk{\Bbbk}

\newcommand\QQ{\mathbb{Q}}
\newcommand\RR{\mathbb{R}}
\newcommand\ZZ{\mathbb{Z}}

\newcommand\floor[1]{\lfloor{#1}\rfloor}

\renewcommand\>{\rangle}

\DeclareMathOperator\GL{GL}
\DeclareMathOperator\perm{perm}

\theoremstyle{plain}% default
\newtheorem{thm}{Theorem}[section]
\newtheorem{prop}[thm]{Proposition}
\newtheorem{lemma}[thm]{Lemma}
\newtheorem{cor}[thm]{Corollary}

\newtheorem*{fthm}{Theorem}

\theoremstyle{definition}
\newtheorem{defn}[thm]{Definition}
\newtheorem{example}[thm]{Example}
\newtheorem{remark}[thm]{Remark}
\newtheorem{question}[thm]{Question}

\usetikzlibrary{arrows,decorations.pathreplacing,shapes}
\newcounter{x}
\newcounter{y}
\newcounter{z}
\newcounter{i}
\newcounter{j}

\newcommand\xaxis{210}
\newcommand\yaxis{-30}
\newcommand\zaxis{90}

\tikzstyle{every circle node}=[draw,minimum size=3pt,inner sep=0]
\tikzstyle{every node}=[transform shape]

\newcommand\topside[4]{
  \fill[fill=white,
  draw=black,shift={(\xaxis:#1)},shift={(\yaxis:#2)},
  shift={(\zaxis:#3)}] (0,0) -- (30:1) -- (0,1) --(150:1)--(0,0);

  \ifnum#4=1
    \draw (0.3,0.5) node (1) [circle, shift={(\xaxis:#1)},shift={(\yaxis:#2)},
    shift={(\zaxis:#3)}, fill=white]{};

    \draw (-0.3,0.5) node (2) [circle, shift={(\xaxis:#1)},shift={(\yaxis:#2)},
    shift={(\zaxis:#3)}, fill=black]{};

    \draw (1) -- (2);
  \fi
}

\newcommand\leftside[4]{
  \fill[fill=gray, draw=black,shift={(\xaxis:#1)},shift={(\yaxis:#2)},
  shift={(\zaxis:#3)}] (0,0) -- (0,-1) -- (210:1) --(150:1)--(0,0);

  \ifnum#4=1

    \draw (-0.6,0.0) node (1) [circle, shift={(\xaxis:#1)},shift={(\yaxis:#2)},
    shift={(\zaxis:#3)}, fill=white]{};

    \draw (-0.3,-0.5) node (2) [circle, shift={(\xaxis:#1)},shift={(\yaxis:#2)},
    shift={(\zaxis:#3)}, fill=black]{};

    \draw (1) -- (2);
  \fi
}

\newcommand\rightside[4]{
  \fill[fill=lightgray, draw=black,shift={(\xaxis:#1)},shift={(\yaxis:#2)},
  shift={(\zaxis:#3)}] (0,0) -- (30:1) -- (-30:1) --(0,-1)--(0,0);

  \ifnum#4=1
    \draw (0.3,-0.5) node (1) [circle, shift={(\xaxis:#1)},shift={(\yaxis:#2)},
    shift={(\zaxis:#3)}, fill=white]{};

    \draw (0.6,0.0) node (2) [circle, shift={(\xaxis:#1)},shift={(\yaxis:#2)},
    shift={(\zaxis:#3)}, fill=black]{};

    \draw (1) -- (2);
  \fi
}

\newcommand\cube[4]{
  \topside{#1}{#2}{#3}{#4} \leftside{#1}{#2}{#3}{#4} \rightside{#1}{#2}{#3}{#4}
}

\newcommand\pp[2]{
 \setcounter{x}{-1}
  \foreach \a in {#1} {
    \addtocounter{x}{1}
    \setcounter{y}{-1}
    \foreach \b in \a {
      \addtocounter{y}{1}
      \setcounter{z}{-1}
      \foreach \c in {1,...,\b} {
        \addtocounter{z}{1}
        \cube{\value{x}}{\value{y}}{\value{z}}{#2}
      }
    }
  }
}

\newcommand\invrside[4]{
  \fill[fill=white, draw=black,shift={(\xaxis:#1)},shift={(\yaxis:#2)},
  shift={(\zaxis:#3)}] (0,0) -- (0,1) -- (30:1) -- (-30:1)--(0,0);

  \ifnum#4 = 1
    \draw (0.3,0.5) node (1) [circle, shift={(\xaxis:#1)},shift={(\yaxis:#2)},
    shift={(\zaxis:#3)}, fill=white]{};

    \draw (0.6,0.0) node (2) [circle, shift={(\xaxis:#1)},shift={(\yaxis:#2)},
    shift={(\zaxis:#3)}, fill=black]{};

    \draw (1) -- (2);
  \fi

  \ifnum#4 = 2
    \draw[thick,white,shift={(\xaxis:#1)},shift={(\yaxis:#2)},
    shift={(\zaxis:#3)}] (0,0) -- (0,1) -- (30:1) -- cycle;
    \draw[shift={(\xaxis:#1)},shift={(\yaxis:#2)},
    shift={(\zaxis:#3)}] (0,0) -- (0,1) -- (30:1) -- cycle;
  \fi
}

\newcommand\invlside[4]{
  \fill[fill=white, draw=black,shift={(\xaxis:#1)},shift={(\yaxis:#2)},
  shift={(\zaxis:#3)}] (0,0) -- (0,1) -- (150:1) --(210:1)--(0,0);

  \ifnum#4 = 1
    \draw (-0.6,0.0) node (1) [circle, shift={(\xaxis:#1)},shift={(\yaxis:#2)},
    shift={(\zaxis:#3)}, fill=white]{};

    \draw (-0.3,0.5) node (2) [circle, shift={(\xaxis:#1)},shift={(\yaxis:#2)},
    shift={(\zaxis:#3)}, fill=black]{};

    \draw (1) -- (2);
  \fi

  \ifnum#4 = 2
    \draw[thick, white,shift={(\xaxis:#1)},shift={(\yaxis:#2)},
    shift={(\zaxis:#3)}] (0,0) -- (0,1) -- (150:1) -- cycle;
    \draw[shift={(\xaxis:#1)},shift={(\yaxis:#2)},
    shift={(\zaxis:#3)}] (0,0) -- (0,1) -- (150:1) -- cycle;
  \fi
}

\newcommand\invdside[4]{
  \fill[fill=white, draw=black,shift={(\xaxis:#1)},shift={(\yaxis:#2)},
  shift={(\zaxis:#3)}] (0,0) -- (-30:1) -- (0,-1) --(-150:1)--(0,0);

  \ifnum#4 = 1
    \draw (0.3,-0.5) node (1) [circle, shift={(\xaxis:#1)},shift={(\yaxis:#2)},
    shift={(\zaxis:#3)}, fill=white]{};

    \draw (-0.3,-0.5) node (2) [circle, shift={(\xaxis:#1)},shift={(\yaxis:#2)},
    shift={(\zaxis:#3)}, fill=black]{};

    \draw (1) -- (2);
  \fi

  \ifnum#4 = 2
    \draw[thick,white,shift={(\xaxis:#1)},shift={(\yaxis:#2)},
    shift={(\zaxis:#3)}] (0,0) -- (-30:1) -- (0,-1) -- cycle;
    \draw[shift={(\xaxis:#1)},shift={(\yaxis:#2)},
    shift={(\zaxis:#3)}] (0,0) -- (-30:1) -- (0,-1) -- cycle;
  \fi
}

\newcommand\outline[4]{
  \foreach \a in {0,...,#1} {
    \foreach \b in {0,...,#3} {
      \invrside{0}{\a}{\b}{#4};
    }
  }

  \foreach \a in {0,...,#2} {
    \foreach \b in {0,...,#3} {
      \invlside{\a}{0}{\b}{#4};
    }
  }

  \foreach \a in {0,...,#1} {
    \foreach \b in {0,...,#2} {
      \invdside{\b}{\a}{0}{#4};
    }
  }
}

\newcommand\young[3]{
  \setcounter{i}{1}
  \foreach \a in #1 {
    \addtocounter{i}{-1}
    \setcounter{j}{-1}
    \foreach \b in {1,...,\a} {
      \addtocounter{j}{1}

      \ifnum #2 = 1
        \draw[scale=0.5,white, fill=white] (\value{j}+#3, \value{i}) rectangle (\value{j}+#3+1,
        \value{i}-1);
        \draw[scale=0.5,dotted, fill=white] (\value{j}+#3, \value{i}) rectangle (\value{j}+#3+1,
        \value{i}-1);
        \ifnum \b = 1
          \draw [scale=0.5] (\value{j}+#3, \value{i}) -- (\value{j}+#3,
          \value{i}-1);
        \fi
      \else
        \draw[scale=0.5,thick, white, fill=white] (\value{j}+#3, \value{i}) rectangle (\value{j}+#3+1,
        \value{i}-1);
      \fi

      \ifnum #2 = 0
        \draw [scale=0.5] (\value{j}+#3+1, \value{i}) -- (\value{j}+#3+1, \value{i}-1)
        -- (\value{j}+#3, \value{i}-1);
      \fi

      \ifnum #2 = 2
        \draw [scale=0.5] (\value{j}+#3, \value{i}) rectangle (\value{j}+1,
        \value{i}-1);
      \fi

    }
  }
}

\newcommand\skewyoung[3]{
  \young{#1}{2}{0}
  \young{#2}{#3}{0}
}

\begin{document}%%%%%%%%%%%%%%%%%%%%%%%%%%%
%%%%%%%%%%%%%%%%%%%%%%%%%%%%%%%%%%%%%%%%%%%

\title{Trivariate monomial complete intersections and plane
  partitions}

\author{Charles Chen}
\email{charleschen@berkeley.edu}

\author{Alan Guo}
\email{aguo@mit.edu}

\author{Xin Jin}
\email{jinxx150@umn.edu}

\author{Gaku Liu}
\email{xueliu@princeton.edu}

\begin{abstract}
  We consider the homogeneous components $U_r$ of the map on $R =
  \kk[x,y,z]/(x^A,y^B,z^C)$ that multiplies by $x+y+z$.  We prove a
  relationship between the Smith normal forms of submatrices of an
  arbitrary Toeplitz matrix using Schur polynomials, and use this to
  give a relationship between Smith normal form entries of $U_r$. We
  also give a bijective proof of an identity proven by J. Li and
  F. Zanello equating the determinant of the middle homogeneous
  component $U_r$ when $(A,B,C) = (a+b,a+c,b+c)$ to the number of
  plane partitions in an $a \times b \times c$ box. Finally, we prove
  that, for certain vector subspaces of $R$, similar identities hold
  relating determinants to symmetry classes of plane partitions, in
  particular classes $3$, $6$, and~$8$.
\end{abstract}

%\date{8 August 2010.}
\maketitle

%%%%%%%%%%%%%%%%%%%%%%%%%%%%%%%%%%%%%%%%%%%
\section{Introduction}%%%%%%%%%%%%%%%%%%%%%
%%%%%%%%%%%%%%%%%%%%%%%%%%%%%%%%%%%%%%%%%%%

For a commutative ring $\kk$ and positive integers $A,B,C$, consider
the trivariate monomial complete intersection $R =
\kk[x,y,z]/(x^A,y^B,z^C)$.  This carries a standard grading in which
$x,y,z$ each have degree one, and decomposes as a direct sum $R =
\bigoplus_{r=0}^e R_r$ where $e := A+B+C-3$, and each homogeneous
component $R_r \cong \kk^{h(r)}$, where $h(r)$ denotes the size of the
set $B_r$ consisting of all monomials of total degree $r$ in $x,y,z$
which are nonzero in $R$.  It is easily seen that
$(h(0),h(1),\ldots,h(e))$ is a symmetric unimodal
sequence. Furthermore, it is known that the maps
$$
U_r : \xymatrix{ R_r \ar[rr]^-{\cdot (x+y+z)} && R_{r+1} }
$$
have $U_{e-r}^t = U_r$, and that $U_r$ is injective for $0 \le r \le
\floor{\frac{e-1}{2}}$ when working with $\kk = \ZZ$ or~$\QQ$ (or, in
fact, with any field of characteristic zero). 

The maps $U_r$ arise in a more general setting in algebraic geometry
and commutative algebra when studying the Weak Lefschetz Property of
general hyperplane sections.  Algebraically, one studies the
multiplication by a general linear form $\ell$ on a graded algebra
$S/I$, where $I$ is a homogeneous ideal in a polynomial ring $S$.  If
$I$ is a monomial ideal, it has been observed in~\cite[Prop~2.2]{MMN}
that choosing $\ell$ as the sum of the variables is enough to
determine if the algebra has the Weak Lefschetz Property.  The paper
focuses on one non-trivial case when one considers the Weak Lefschetz
Property and one that has recently been studied (\cite{LiZanello},
\cite{BrennerKaid}).

Our first main result attempts to address how the maps $U_r$ behave,
by considering the Smith normal form of $U_r$ when working over $\kk =
\ZZ$. We say that a matrix $U$ in $\ZZ^{m \times n}$, $m \ge n$, has
$\textsc{SNF}(U)= (a_1,a_2,\ldots,a_n)$ if there exist matrices $P,Q
\in \GL_m(\ZZ),\GL_n(\ZZ)$ such that $PUQ$ takes the diagonal form
$$
\left(
\begin{array}{ccccc}
a_1 & 0 & \cdots & 0 & 0 \\
0 & a_2 & \cdots & 0 & 0 \\
\vdots & \vdots & \ddots & \vdots & \vdots \\
0 & 0 & \cdots & a_{n-1} & 0 \\
0 & 0 & \cdots & 0 & a_n \\
\hline
0 & 0 & \cdots & 0 & 0 \\
\vdots & \vdots & \ddots & \vdots & \vdots \\
0 & 0 & \cdots & 0 & 0
\end{array}
\right)
$$
and $a_i$ divides $a_{i+1}$.

\begin{thm}\label{t:snf}
  Assume $A \ge B \ge C \ge 1$.
\begin{enumerate}[(i)]
\item%
  For $0 \le r \le A-2$,
  $$
  \textsc{SNF}(U_r) = \underbrace{(1,1,\ldots,1)}_{h(r)}.
  $$
  In particular, these maps $U_r$ are injective when $\kk$ is any
  field.

\item%
  For $A-1 \le r \le \floor{\frac{e-1}{2}}$, the non-unit entries in
  $\textsc{SNF}(U_r)$ are the same as the non-unit entries of
  $\textsc{SNF}(M_r(A,B,C))$ where
  $$
  M_r(A,B,C) := \left( \ {A \choose r-B+i-j+2} \ \right)_{
  \substack{
  i = 1,\ldots,B+C-r-2 \\
  \hspace{-1.35em} j = 1,\ldots,r-A+2
  }}.
  $$
  In particular, there are at most $r-A+2$ such non-unit entries, so
  that for any field~$\kk$, the map $U_r$ has nullity at most $r-A+2$.

\item%
  Let $m := \floor{\frac{e-1}{2}}$, so that $U_m : R_m \to R_{m+1}$ is
  a map closest to the ``middle'' of~$R$. If $\textsc{SNF}(U_m) =
  (a_1,a_2,\ldots,a_{h_m})$, then for $0 \le s \le m$ one has
  $$
  \textsc{SNF}(U_{m-s}) = (\underbrace{1,1,\ldots,1}_{s-h(m)+h(m-s)},
  a_1,a_2,\ldots,a_{h(m)-s}).
  $$
\end{enumerate}
\end{thm}

On the way to proving this theorem, we prove two lemmas.  The first
lemma provides a relationship between the Smith normal forms of
submatrices of an arbitrary Toeplitz matrix using Schur polynomials
(Lemma~\ref{l:toep}), and the second lemma proves an ``inverse''
Littlewood-Richardson rule (Lemma~\ref{l:main}). Both of these lemmas
might be of independent interest.

Our other main results relate to the middle map $U_m$ when $A+B+C$ is
even (and without loss of generality, $A \le B+C$ so that $r = m$ does
not fall in the trivial case (i) of Theorem~\ref{t:snf} above). In
this case, one can check that $h(m) = h(m+1)$, so that $U_m$ is
square. Li and Zanello proved the striking result that $\det(U_m)$, up
to sign, counts the number of plane partitions that fit in an $a
\times b \times c$ box where $a := \frac{A+B-C}{2}$, $b :=
\frac{A-B+C}{2}$, $c := \frac{-A+B+C}{2}$ so that $A = a+b$, $B =
a+c$, $C = b+c$. Their proof proceeded by evaluating
$$
\det(U_m) = \det(M_m(A,B,C))
$$
directly, and comparing the answer to known formulae for such plane
partitions.  We respond to their call for a more direct, combinatorial
explanation (see~\cite{LiZanello}) with the following:

\begin{thm}\label{t:pp}
  Expressed in the monomial $\ZZ$-basis for $R =
  \ZZ[x,y,z]/(x^{a+b},y^{a+c},z^{b+c})$ the map $U_m : R_m \to
  R_{m+1}$ has its determinant $\det(U_m)$ equal, up to sign, to its
  permanent $\perm(U_m)$, and each nonzero term in its permanent
  corresponds naturally to a plane partition in an $a \times b \times
  c$ box.
\end{thm}

The same ideas then allow us to express the counts for other symmetry
classes of plane partitions, namely those which are cyclically
symmetric (class~$3$) or transpose complementary (class~$6$) or
cyclically symmetric and transpose complementary (class~$8$), in terms
of the determinant of $U_m$ when restricted to certain natural
$\ZZ$-submodules of~$R$.

\begin{thm}\label{t:cspp}
  Assuming $a=b=c$, let $C_3 = \ZZ/3\ZZ = \{1,\rho,\rho^2\}$ act on $R
  = \ZZ[x,y,z]/(x^{2a},y^{2a},z^{2a})$ by cycling the variables $x
  \stackrel{\rho}{\to} y \stackrel{\rho}{\to} z \stackrel{\rho}{\to}
  x$. Then the map $U_m|_{R^{C_3}}$ restricted to the $m$-th
  homogeneous component of the $C_3$-invariant subring $R^{C_3}$ has
  $\det(U_m|_{R^{C_3}})$ equal, up to sign, to the number of
  cyclically symmetric plane partitions in an $a \times a \times a$
  box.
\end{thm}

\begin{thm}\label{t:tcpp}
  Assuming $a=b$ and the product $abc$ is even, let $C_2 = \ZZ/2\ZZ =
  \{1,\tau\kappa\}$ act on $R = \ZZ[x,y,z]/(x^{2a},y^{a+c},z^{a+c})$
  by swapping $y \leftrightarrow z$.  Then the map $U_m|_{R^{C_2,-}}$
  restricted to the $m$-th homogeneous component of the anti-invariant
  submodule $R^{C_2,-} := \{f \in R : \tau\kappa(f) = -f\}$ has
  $\det(U_m|_{R^{C_2,-}})$ equal, up to sign, to the number of
  transpose complementary plane partitions in an $a \times a \times c$
  box.
\end{thm}

\begin{thm}\label{t:cstcpp}
  Assuming $a=b=c$ are all even, let $C_2,C_3$ act on $R = \ZZ[x,y,z]/
  (x^{2a},y^{2a},z^{2a})$ as before. Then the map $U_m|_{R^{C_3} \cap
    R^{C_2,-}}$ restricted to the $m$-th homogeneous component of the
  intersection $R^{C_3} \cap R^{C_2,-}$ has $\det(U_m|_{R^{C_3} \cap
    R^{C_2,-}})$ equal, up to sign, to the number of cyclically
  symmetric transpose complementary plane partitions in an $a \times a
  \times a$ box.
\end{thm}

%%%%%%%%%%%%%%%%%%%%%%%%%%%%%%%%%%%%%%%%%%%
\section{Proof of Theorem~\ref{t:snf}}
%%%%%%%%%%%%%%%%%%%%%%%%%%%%%%%%%%%%%%%%%%%

%%%%%%%%%%%%%%%%%%%%%%%%%%%%%%%%%%%%%%%%%%%%%%
\subsection*{Proof of Theorem~\ref{t:snf}, part (i)}%%
%%%%%%%%%%%%%%%%%%%%%%%%%%%%%%%%%%%%%%%%%%%%%%
Recall the statement of Theorem~\ref{t:snf}, part (i):

\begin{fthm}
  For $0 \le r < A-1$, $\textsc{SNF}(U_r) = \{
  \underbrace{1,\ldots,1}_{h(r) \text{ ones}} \}$. In other words, the
  cokernel of $U_r$ is free of rank $h(r+1) - h(r)$.
\end{fthm}

\begin{proof}
  Let $B_s$ denote the monomial basis of $R_s$. We represent $U_r$ by
  a matrix whose columns and rows are indexed by elements of $B_r$ and
  $B_{r+1}$, respectively.  We will prove our claim by showing that,
  for some ordering of the rows and columns of the matrix of $U_r$,
  there exists a lower unitriangular maximal submatrix. To do this, it
  suffices to show that there exists an ordering of the columns such
  that, for each column $j$, there exists a row $i$ such that the
  $(i,j)$ entry is $1$ and the entries to the right are~$0$. We
  arrange the columns lexicographically, so that monomials with higher
  $x$ power are on the left, and if $x$ powers are equal, then we use
  the power of $y$ to break ties, and then the power of $z$ to break
  remaining ties. For example, for $A=B=C=4$ and $r = 4$, the
  monomials would be ordered
$$
x^3y \ \ x^3z \ \ x^2y^2 \ \ x^2yz \ \ x^2z^2 \ \ xy^3 \ \ xy^2z \ \
xyz^2 \ \ xz^3 \ \ y^3z \ \ y^2z^2 \ \ yz^3.
$$
For any given monomial $x^iy^jz^k$ indexing a column, the monomial
$x^{i+1}y^jz^k$ is nonzero in the quotient ring
$\kk[x,y,z]/(x^A,y^B,z^C)$ since $i+1 \le i+j+k+1 < A$. Therefore
$x^{i+1}y^jz^k$ indexes a row. The entry of this row in column
$x^iy^jz^k$ is $1$, and any other column with a $1$ in this row must
be indexed by either $x^{i+1}y^{j-1}z^k$ or $x^{i+1}y^jz^{k-1}$, both
of which lie to the left of $x^iy^jz^k$.
\end{proof}

\begin{remark}
  From the previous proposition, it follows that all the maps $U_r$ of
  $\kk[x,y]/(x^A, y^B)$ (setting $C = 1$) and $\kk[x]/(x^A)$ (setting
  $B = C = 1$) have free cokernel.  Note that this immediately proves
  the Weak Lefschetz Property for (essentially all) codimension $2$
  monomial Artinian complete intersections, regardless of
  characteristic.  (To be precise, we require the base field to be
  infinite.  We can define the Weak Lefschetz Property over finite
  fields, but it becomes too pathological to be of interest.)  More
  generally, it turns out that the WLP holds for any codimension $2$
  standard graded Artinian algebra \cite{MigZan}.

  In the other extreme (with an arbitrary number of variables and
  bounded powers), Wilson completely determined the cokernel of the
  maps $U_r$ for $\kk[x_1,\ldots,x_n]/(x_1^2,\ldots,x_n^2)$
  \cite{Wilson}.  Moreover, Hara and Watanabe provided an elementary
  proof of Wilson's result and used it to show the Strong Lefschetz
  Property for $\kk[x_1,\ldots,x_n]/(x_1^2,\ldots,x_n^2)$ over certain
  fields \cite{HW}.  We warn the reader that Hara and Watanabe's
  result is not characteristic free.  For instance, when $n = 2$, the
  SLP result for $\kk[x_1, x_2]/(x_1^2, x_2^2)$ fails in
  characteristic $2$, as multiplication by $(x+y)^2$ is not injective.
\end{remark}

Things get interesting when $A-1 \le r \le \frac{e}{2}$. The up maps
are no longer necessarily injective over arbitrary fields. However, we
can still prove a least upper bound on the number of non-unit elements
in $\textsc{SNF}(U_r)$. In fact, we have the following

%%%%%%%%%%%%%%%%%%%%%%%%%%%%%%%%%%%%%%%%%%%%%%
\subsection*{Proof of Theorem~\ref{t:snf}, part (ii)}%%
%%%%%%%%%%%%%%%%%%%%%%%%%%%%%%%%%%%%%%%%%%%%%%
Recall the statement of Theorem~\ref{t:snf}, part (ii):

\begin{fthm}
  For $A-1 \le r \le \floor{\frac{e-1}{2}}$, the non-unit entries in
  $\textsc{SNF}(U_r)$ are the same as the non-unit entries of
  $\textsc{SNF}(M_r(A,B,C))$ where
  \begin{equation}\label{eq:MR}
  M_r(A,B,C) := \left( \ {A \choose r-B+i-j+2} \ \right)_{
  \substack{
  i = 1,\ldots,B+C-r-2 \\
  \hspace{-1.35em} j = 1,\ldots,r-A+2
  }}.
  \end{equation}
  In particular, there are at most $r-A+2$ such non-unit entries, so
  that for any field~$\kk$, the map $U_r$ has nullity at most $r-A+2$.
\end{fthm}

\begin{proof}
  We construct a matrix $U_r(y,z)$ as follows. First, represent $U_r$
  as a matrix in the monomial basis with a lexicographic ordering on
  its rows and columns, as in the previous proof.  For each $1$ in
  $U_r$, if the row index divided by the column index equals $y$,
  replace the $1$ with $y$, and similarly for $z$. Then move all the
  columns indexed by monomials with $x^{A-1}$ to the right,
  maintaining the ordering amongst these columns. There are $r-A+2$
  such columns: for $x^{A-1}y^jz^k$, with $j+k=r-A+1$, note that since
  $r \le \frac{A+B+C-3}{2}$ implies $r-A+1 \le \frac{B+C-A-1}{2} \le
  \frac{C-1}{2} < B-1,C-1$, $j$ can take on every value from 0 to
  $r-A+1$.  Now we claim that there are $B+C-r-2$ monomials of rank
  $r+1$ with no powers of $x$. First, note that $r \le
  \frac{A+B+C-3}{2}$ implies $B+C-r-2 \ge r-(A-1) \ge 0$. Now, the
  monomials of rank $r+1$ with no powers of $x$ are precisely
  $y^{C-1}z^{r+2-C}, y^{C-2}z^{r+3-C}, \ldots, y^{r+2-B}z^{B-1}$.
  
 We now have a matrix of the form
$$
U_r(x,y) = \left(
\begin{array}{ccc|ccc}
1 &    & 0 &   &    &    \\
    & \ddots &    &   & *  &  \\
 * &   & 1 & & & \\
 \hline
 & & & & & \\
 &*& & & 0 & \\
 & & & & &
\end{array}
\right)
$$
where the lower right block of zeroes has $B+C-r-2$ rows and $r-A+2$
columns.  Now we perform the following algorithm: initialize $M :=
U_r(y,z)$. While
$$
M = \left(
\begin{array}{c|ccc}
1 &  & Y & \\
\hline
 & & & \\
 X & & M' & \\
  & & &
\end{array}
\right)
$$
and $M$ has more than $r-A+2$ rows, set $M := M' - XY$. Each step of
the algorithm is using the $1$ in the first row and first column as a
pivot and performing $\ZZ$-invertible row and column operations on $M$
to eliminate the other entries in the same row and same column. Then
we focus on the remaining submatrix $M'$ and repeat.

\begin{lemma}
  At each step of the algorithm, the following holds: fix a monomial
  $\alpha$ of total degree $r-A+1$ with no powers of $x$, fix a
  nonnegative integer $\gamma$ and a monomial $\beta$ of total degree
  $r+1-\gamma$ with no powers of $x$, and let $\mu$ be the entry of
  $M$ whose column is indexed by $x^{A-1}\alpha$ and whose row is
  indexed by $x^\gamma\beta$, and suppose $\beta/\alpha = y^jz^k$. If
  no entry to the left of $\mu$ indexed by a monomial whose power of
  $x$ is less than $A-1$ is a $y$ or $z$, then $\mu =
  (-1)^{A-\gamma-1}{j+k \choose j} y^jz^k$.
\end{lemma}
\begin{proof}
  This statement is true at the beginning of the algorithm. We now
  proceed inductively.  Suppose $\mu$ is as in the statement of the
  Lemma. Suppose $yz$ divides $\beta$.  For some previous step of the
  algorithm, there existed an entry $\mu_y$ whose column was indexed
  by $x^{A-1}\alpha$ and whose row was indexed by
  $x^{\gamma+1}\beta/y$, and no entry to the left of $\mu_y$ indexed
  by something with $x$ power less than $A-1$ was a $y$ or $z$.  We
  have the following setup:
$$
\bordermatrix{%
 & x^\gamma \beta/y & \cdots & x^{A-1}\alpha & \cdots \cr
x^{\gamma+1}\beta/y & 1 & \cdots & \mu_y & \cdots \cr
\vdots & \vdots & \ddots & \vdots  \cr
x^\gamma \beta & y & \cdots
}.
$$
Similarly, for some previous step of the algorithm, there existed an
entry $\mu_z$ whose column was indexed by $x^{A-1}\alpha$ and whose
row was indexed by $x^{\gamma+1}\beta/z$, and no entry to the left of
$\mu_z$ indexed by something with $x$ power less than $A-1$ was a $y$
or $z$. We have an analogous setup:
$$
\bordermatrix{%
 & x^\gamma \beta/z & \cdots & x^{A-1}\alpha & \cdots \cr
x^{\gamma+1}\beta/z & 1 & \cdots & \mu_z & \cdots \cr
\vdots & \vdots & \ddots & \vdots  \cr
x^\gamma \beta & z & \cdots
}.
$$
By the algorithm, after both of these steps, we end up with $\mu =
-y\mu_y - z\mu_z$.  But by induction,
\begin{eqnarray*}
  -y\mu_y - z\mu_z &=& -y(-1)^{A-\gamma-2}{j+k-1 \choose j-1}
  y^{j-1}z^k
  - (-1)^{A-\gamma-2}{j+k-1 \choose j} y^jz^{k-1} \\
  &=& (-1)^{A-\gamma-1} \left[ {j+k-1 \choose j-1} + {j+k-1 \choose j}
  \right] y^jz^k \\
  &=& (-1)^{A-\gamma-1} {j+k \choose j} y^jz^k.
\end{eqnarray*}
In the case where $yz$ does not divide $\beta$, at least one of these
previous steps did not exist, but our argument still applies and the
result still follows.
\end{proof}

In particular, the Lemma implies that, at the end of the algorithm,
the matrix $M$ has $(B+C-r-2)$ rows and $(r-A+2)$ columns and takes
the form
$$
\left( (-1)^{A-1} {A \choose r-B+i-j+2} y^{A+B-r-i+j-2} z^{r-B+i-j+2}
\right)_{i,j}.
$$
Recall that at each step of the algorithm, we are simply performing a
sequence of $\ZZ$-invertible row and column operations. Therefore
$U_r(y,z)$ and $(-1)^{A-1}M$ have the same non-unit Smith normal form
entries. Substituting $y = z = 1$ gives the result.
\end{proof}

\begin{remark}
  Note that after substituting $A = a+b$, $B = a+c$, $C = b+c$, we
  return to the context of the plane partitions. In this case, our
  matrix $M_r(A,B,C)$ becomes an $(r-a-b+2) \times (a+b+2c-r-2)$
  matrix with entries
\begin{equation}\label{eq:Mr}
M_r(a+b,a+c,b+c) = \left( {a+b \choose r-a-c+i-j+2} \right)_{i,j}.
\end{equation}
Focusing on the middle rank $r = a+b+c-2$ yields a $c \times c$ matrix
$$
M_r(a+b,a+c,b+c) = \left( {a+b \choose b+i-j} \right)_{i,j}.
$$

The matrix $M_{a+b+c-2}(a+b,a+c,b+c)$ is called a \emph{Carlitz
  matrix} \cite{Kup98}, and occurs in \cite[Lemma~2.2]{LiZanello} and
\cite[Theorem~4.3]{CookNagel}.  The Smith normal form of these
matrices is not known, but Kuperberg conjectures that there is a
potential combinatorial connection between plane partitions and Smith
forms of Carlitz matrices \cite{Kup02}.  Even for small numbers,
however, it is subtle to understand (and even to compute!)  the Smith
forms.  We give some small examples:

When $c = 1$, the only Smith entry is $\binom{a+b}{b}$.  When $c = 2$
and $a = b$, the explicit row and column operations to turn
$M_{a+b+c-2}(a+b,a+c,b+c)$ into Smith normal form are
\[
\begin{pmatrix} 1 & -1 \\ -1 - 3a & 2 + 3a \end{pmatrix}
M_{a+b+c-2}(a+b,a+c,b+c)
\begin{pmatrix} 2 & 1 \\ 1 & 1 \end{pmatrix}
=
\begin{pmatrix} \frac{\binom{2a}{a}}{a+1} & 0 \\ 0 &
  \binom{2a+1}{a+1} \end{pmatrix}.
\]

Even though the first entry is the $a$-th Catalan number, a
combinatorial explanation of the Smith entries eludes the authors.

Directly calculating the Smith entries by computing GCDs of various
binomial coefficients seems intractable and does not seem to
generalize for bigger $c$.  In particular, the direct method seems
hard to understand even when $c = 2$ and $a$, $b$ are arbitrary.  In
that case, the Smith entries are $s_1 = \gcd\left(\binom{a+b}{b-1},
  \binom{a+b}{b}, \binom{a+b}{b+1}\right)$ and $s_2 =
\frac{\binom{a+b}{b}^2 - \binom{a+b}{b-1} \binom{a+b}{b+1}}{s_1}$.

\end{remark}

\begin{remark}
  If $R = \kk[x_1,\ldots,x_n]/(x_1^{A_1},\ldots,x_n^{A_n})$, then we
  can easily generalize the proof above to show that the non-unit
  Smith entries of the map $U_r$ (which is now defined by
  multiplication by $x_1 + \cdots + x_n$) for $A_1 - 1 \le r \le
  \frac{A_1 + \cdots + A_n - n}{2}$ are the same as those of the
  matrix with the following entries, assuming $A_1 \ge A_2 \ge \cdots
  \ge A_n$: if the column is indexed by $x_1^{A_1-1}\alpha$ and the
  row is indexed by $\beta$ where $x$ divides neither $\alpha$ nor
  $\beta$, and $\beta/\alpha = x_2^{i_2} \cdots x_n^{i_n}$, then the
  entry is the multinomial coefficient
$$
{ A_1 \choose i_2,i_3,\ldots,i_n } = \frac{A_1!}{i_2!i_3! \cdots i_n!}.
$$
Computer evidence suggests that the non-unit Smith entries of these
matrices behave nicely for $n = 4$ (as in $\textsc{SNF}(U_r)$ is a
submultiset of $\textsc{SNF}(U_{r+1}))$, but the analogous result is
unfortunately not true for $n = 5$: taking $A_1 = A_2 = A_3 = A_4 =
A_5 = 4$, the Smith entry $70$ occurs in $\textsc{SNF}(U_6)$ but not
in $\textsc{SNF}(U_7)$.
\end{remark}

Letting $s(r)$ denote the number of non-unit Smith normal form entries
of $U_r$ and $m = a+b+c-2$, Theorem~\ref{t:snf}(ii) implies that
$s(m-i) \le c-i$ for all $i \le c$. In fact, something stronger holds,
which is Theorem~\ref{t:snf}(iii).
%%%%%%%%%%%%%%%%%%%%%%%%%%%%%%%%%%%%%%%%%%%%%%
\subsection*{Proof of Theorem~\ref{t:snf}, part (iii)}%%
%%%%%%%%%%%%%%%%%%%%%%%%%%%%%%%%%%%%%%%%%%%%%%
Recall the statement of Theorem~\ref{t:snf}, part (iii):

\begin{fthm}
  Let $m := \floor{\frac{e-1}{2}}$, so that $U_m : R_m \to R_{m+1}$ is
  a map closest to the ``middle'' of~$R$. If $\textsc{SNF}(U_m) =
  (a_1,a_2,\ldots,a_{h(m)})$, then for $0 \le s \le m$ one has
  $$
  \textsc{SNF}(U_{m-s}) = (\underbrace{1,1,\ldots,1}_{s-h(m)+h(m-s)},
  a_1,a_2,\ldots,a_{h(m-s)}).
  $$
\end{fthm}

Immediately from part~(iii)~of~Theorem~\ref{t:snf}, we re-derive a
special case of~\cite[Prop~2.1(b)]{MMN}:

\begin{cor}
The maps $U_r$ for $r \le m$ are injective if and only if $U_m$ is
injective.
\end{cor}

Recall the matrices $M_r(A,B,C)$ given by~\eqref{eq:MR}. The first
one, $M_{A-1}(A,B,C)$, is a matrix with $1$ column and $-A+B+C-1$ rows,
and in general, the $i$-th matrix has $i$ columns and $-A+B+C-i$ rows. We
observe that in fact these are all submatrices of the $(-A+B+C-1) \times
(-A+B+C-1)$ lower triangular Toeplitz matrix with entries
$$
\left( {A \choose A-B+i-j+1} \right)_{i,j}
$$
for $i \ge j$, and $0$ otherwise, which, written out, looks like
$$
\begin{pmatrix}
{A \choose A-B+1} & 0 & \cdots & 0 & 0 \\
{A \choose A-B+2} & {A \choose A-B+1} & \cdots & 0 & 0 \\
\vdots & \vdots & \ddots & \vdots & \vdots \\
{A \choose C-1} & {A \choose C-2} & \cdots & {A \choose
  A-B+2}
& {A \choose A-B+1}
\end{pmatrix}.
$$
For $1\le i \le \frac{-A+B+C}{2}$, the matrix $M_{A-2+i}(A,B,C)$ is
simply the submatrix of the Toeplitz matrix created by choosing the
first $i$ columns and the last $-A+B+C-i$ rows.  Surprisingly, there
is nothing special about the entries of the large Toeplitz matrix!  We
have the following more general result, which immediately implies
part~(iii)~of~Theorem~\ref{t:snf}:

\begin{lemma}\label{l:toep}
Let $A$ be an arbitrary $n \times n$ Toeplitz matrix,
$$
A = \begin{pmatrix}
h_n & & & & & & \\
h_{n-1} & h_n & & & & & \\
h_{n-2} & h_{n-1} & h_n & & & & \\
\vdots & \vdots & \vdots & \ddots & & & \\
h_3 & h_4 & h_5 & \cdots & h_n & & \\
h_2 & h_3 & h_4 & \cdots & h_{n-1} & h_n & \\
h_1 & h_2 & h_3 & \cdots & h_{n-2} & h_{n-1} & h_n
\end{pmatrix}
$$
with entries in a principal ideal domain, and for $1 \le c \le n$, let
$A_c$ denote the $(n-c+1) \times c$ submatrix of $A$ formed by columns
$1,\ldots,c$ and rows $c,\ldots,n$. Then, for $1 \le k \le \frac n2$,
the $k$-th Smith normal form entry of $A_c$ is the same for all $k \le
c \le \frac n2$.
\end{lemma}

Our proof of Lemma~\ref{l:toep} requires a bit of algebraic machinery,
so we defer it to its own section.
%%%%%%%%%%%%%%%%%%%%%%%%%%%%%%%%%%%%%%%%%%%
\section{Proof of Lemma~\ref{l:toep}}%%%%%%
%%%%%%%%%%%%%%%%%%%%%%%%%%%%%%%%%%%%%%%%%%%
Suppose $M$ is a matrix over a PID with Smith normal form entries $a_1
\le a_2 \le \cdots \le a_r$. Then it is known (see~\cite{Newman}) that
$$
a_k = \frac{\gcd(\text{$k \times k$ minors of $M$})}{\gcd(\text{$(k-1)
    \times (k-1)$ minors of $M$})}.
$$
Therefore, to prove Lemma~\ref{l:toep}, it suffices to show that,
for $1 \le k \le c \le \frac n2$, the ideal generated by the $k \times
k$ minors of $A_c$ is equal to the ideal generated by the $k \times k$
minors of $A_k$.

\begin{example}
Suppose $n=7$. The matrix $A$ is
$$
A = \begin{pmatrix}
h_7 \\
h_6 & h_7 \\
h_5 & h_6 & h_7 \\
h_4 & h_5 & h_6 & h_7 \\
h_3 & h_4 & h_5  & h_6 & h_7 \\
h_2 & h_3 & h_4 & h_5 & h_6 & h_7 \\
h_1 & h_2 & h_3 & h_4 & h_5 & h_6 & h_7
\end{pmatrix}
$$
and
$$
A_1 = \begin{pmatrix}
h_7 \\
h_6 \\
h_5 \\
h_4 \\
h_3 \\
h_2 \\
h_1
\end{pmatrix}, \
A_2 = \begin{pmatrix}
h_6 & h_7 \\
h_5 & h_6 \\
h_4 & h_5 \\
h_3 & h_4 \\
h_2 & h_3 \\
h_1 & h_2
\end{pmatrix}, \
A_3 = \begin{pmatrix}
h_5 & h_6 & h_7 \\
h_4 & h_5 & h_6 \\
h_3 & h_4 & h_5 \\
h_2 & h_3 & h_4 \\
h_1 & h_2 & h_3
\end{pmatrix}, \
A_4 = \begin{pmatrix}
h_4 & h_5 & h_6 & h_7 \\
h_3 & h_4 & h_5 & h_6 \\
h_2 & h_3 & h_4 & h_5 \\
h_1 & h_2 & h_3 & h_4
\end{pmatrix}.
$$
The case $k = 1$ is trivial because the $1 \times 1$ minors of the
matrices are just the entries themselves. The $k=2$ case is slightly
more complicated. For example, the minor
$$
\begin{vmatrix}
h_6 & h_7 \\
h_1 & h_2
\end{vmatrix}
$$
can be found in $A_2$ but not $A_3$; however, we can write
$$
\begin{vmatrix}
h_6 & h_7 \\
h_1 & h_2
\end{vmatrix} = 
\begin{vmatrix}
h_5 & h_7 \\
h_1 & h_3
\end{vmatrix} -
\begin{vmatrix}
h_5 & h_6 \\
h_2 & h_3
\end{vmatrix},
$$
where the minors on the right hand side can be found in $A_3$. One can
see that trying to do this systematically for $k \times k$ minors with
$k \ge 3$ gets rather difficult.
\end{example}

Since minors are rather difficult to work with directly, we instead
use Schur polynomials.  The \emph{complete homogeneous symmetric
  polynomials} in $n$ variables $x_1,\ldots,x_n$ are the polynomials
\begin{eqnarray*}
h_0(x_1,\ldots,x_n) &=& 1 \\
h_1(x_1,\ldots,x_n) &=& \sum_{1 \le i \le n} x_i \\
h_2(x_1,\ldots,x_n) &=& \sum_{1 \le i \le j \le n} x_ix_j \\
h_3(x_1,\ldots,x_n) &=& \sum_{1 \le i \le j \le k \le n} x_ix_jx_k \\
 &\vdots &
\end{eqnarray*}
where $h_d(x_1,\ldots,x_n)$ is the sum of all monomials of total
degree $d$.  What these polynomials actually are is not so important
to our proof; the importance lies in the fact that $h_1,\ldots,h_n$
are algebraically independent. An \emph{integer partition} $\lambda =
(\lambda_1,\lambda_2,\ldots)$ of $n$ is a non-increasing sequence of
nonnegative integers $\lambda_i$ such that $\sum_i \lambda_i = n$.  If
$k$ is the largest number such that $\lambda_k > 0$, then we say
$\lambda$ \emph{has $k$ parts}.  If $\lambda_i \le \ell$, we say that
the $i$-th part of $\lambda$ is at most $\ell$.  If $\lambda$ has $k$
parts and $\lambda_1 \le \ell$, then we say that $\lambda$ \emph{fits
  in a $k \times \ell$ box}. We associate to each integer partition
$\lambda$ a Young diagram (see Figure~\ref{fig:partition}).
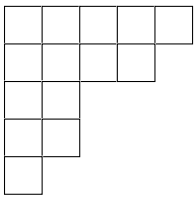
\begin{figure}[htbp]
\begin{tikzpicture}
\young{{5,4,2,2,1}}{2}{0}
\end{tikzpicture}
\caption{Young diagram representing $(5,4,2,2,1)$}
\label{fig:partition}
\end{figure}

If $\lambda$ and $\mu$ are two partitions and $\mu_i \le \lambda_i$
for all $i$, then we say $\mu$ is \emph{contained in} $\lambda$. If
so, then we associate a skew diagram to $\lambda$ \emph{mod} $\mu$,
written $\lambda / \mu$ (see Figure~\ref{fig:mod}).
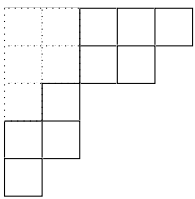
\begin{figure}[htbp]
\begin{tikzpicture}
\skewyoung{{5,4,2,2,1}}{{2,2,1}}{1};
\draw[very thick, white,scale=0.5] (0,-0) -- (0,-3);
\draw[dotted,scale=0.5] (0,-0) -- (0,-3);
\draw[scale=0.5] (0,-3) -- (1,-3) -- (1,-2) -- (2,-2) -- (2,0);
\end{tikzpicture}
\caption{Skew diagram representing $(5,4,2,2,1) / (2,2,1,0,0)$}
\label{fig:mod}
\end{figure}

To any such $\lambda / \mu$, we associate the \emph{Schur polynomial}
\begin{equation}\label{eq:schur}
  S_{\lambda / \mu} = \det\left( h_{\lambda_i - \mu_j - i + j} \right).
\end{equation}
where we take $\mu_k = 0$ if $\lambda$ has $k$ parts.  Schur
polynomials correspond to minors found in Toeplitz matrices, which is
precisely what we need.  For a more comprehensive treatment of Schur
polynomials and Young diagram, we refer the reader to~\cite{Fulton}. A
Schur polynomial $S_{\lambda / \mu}$ is \emph{non-skew} if $\mu = 0$,
and \emph{skew} otherwise.

Our primary weapon of attack will be the \emph{Littlewood-Richardson
  rule}, which tells us how to decompose a skew Schur polynomial into
a linear combination of non-skew Schur polynomials (with positive
coefficients, even!).  Suppose $S_{\lambda/\mu}$ is a skew Schur
polynomial. Draw the diagram associated with $\lambda / \mu$. A
\emph{labeling} of $\lambda/\mu$, where $\lambda$ has $k$ parts, is
defined as a labeling of the squares of the diagram with integers
$1,\ldots,k$ such that:
\begin{itemize}
  \item%
  the numbers are weakly increasing from left to right in each row;
  \item%
  the numbers are strictly increasing from top to bottom in each
  column;
\item%
  when reading the string of the numbers from right to left in each
  row, top row to bottom row, each initial substring must have at
  least as many $i$'s as $i+1$'s, for all $i$.
\end{itemize}
The first two properties say that the labeling forms a skew
semistandard Young tableau.  A partition $\pi$ \emph{arises from a
  labeling of} $\lambda$ if, for each $i$, the $i$-th part of $\pi$
has size equal to the number of $i$'s in the labeling of
$\lambda$. The Littlewood-Richardson rule states that
$$
S_{\lambda / \mu} = \sum_{\pi} S_{\pi}
$$
where the sum is over all partitions $\pi$ that arise from a labeling
of $\lambda / \mu$ (see Figure~\ref{fig:LR} for an illustration of
this rule).
\begin{figure}[htbp]
\begin{tikzpicture}
\matrix[column sep=6,row sep=3] {
\skewyoung{{5,3}}{{1}}{0}&
\draw (1.25,-0.5) node {$=$};&
\skewyoung{{5,3}}{{1}}{0}
\foreach \a in {1,...,4}
  \draw (\a/2 + 0.25, -0.25) node {$1$};
\draw (0.25, -0.75) node {$1$};
\foreach \a in {1,...,2}
  \draw (\a/2 + 0.25, -0.75) node {$2$};
&
\draw (1.25,-0.5) node {$+$};&
\skewyoung{{5,3}}{{1}}{0}
\foreach \a in {1,...,4}
  \draw (\a/2 + 0.25, -0.25) node {$1$};
\foreach \a in {0,...,2}
  \draw (\a/2 + 0.25, -0.75) node {$2$};
\\
&
\draw (1.25,-0.5) node {$=$};&
\young{{5,2}}{2}{0}&
\draw (1.25,-0.5) node {$+$};&
\young{{4,3}}{2}{0}\\
};
\end{tikzpicture}
\caption{The Littlewood-Richardson rule showing
$S_{(5,3)/(1,0)} = S_{(5,2)} + S_{(4,3)}$.}
\label{fig:LR}
\end{figure}
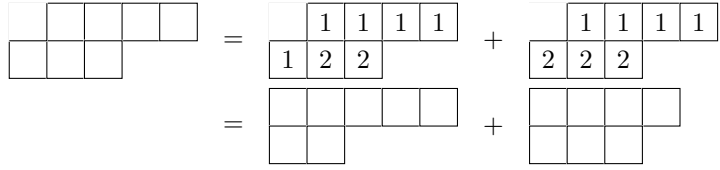

Since the Schur polynomials $h_1,\ldots,h_n$ in $n$ variables are
algebraically independent, we may treat the Littlewood-Richardson rule
as an algebraic identity for $S_{\lambda/\mu}$ of the form
\eqref{eq:schur}, where the $h_i$ are formal variables.

Recall our definition of the matrix $A$ and its submatrices $A_c$ in
the statement of the Lemma. We now introduce a new term:

\begin{defn}
  A Schur polynomial $S_{\lambda/\mu}$ is \emph{$(k,c)$-legal} if it
  is equal to a $k \times k$ minor in $A_c$. We also say that
  $\lambda/\mu$ is $(k,c)$-legal, if $S_{\lambda/\mu}$ is
  $(k,c)$-legal, identifying the polynomial with the diagram.
\end{defn}

By this definition, our task amounts to proving that the ideal
generated by $(k,k)$-legal diagrams is equal to the ideal generated by
$(k,c)$-legal diagrams, for all $1 \le k \le \frac n2$ and for all $k
\le c \le \frac n2$. To make our job easier, we have the following
characterization:

\begin{prop}\label{p:legal}
  A skew diagram $\lambda/\mu$ with $k$ parts is $(k,c)$-legal if and
  only if the following hold:
\begin{itemize}
  \item%
    $\lambda_1 \le n-k+1$;
  \item%
    $\lambda_k \ge k$;
  \item%
    $\mu_1 \le c-k$;
  \item%
    $\lambda_i - \mu_i \ge k$ for all $i$;
  \item%
    $\lambda_1 - \lambda_k \le n-c-k+1$.
\end{itemize}
\end{prop}
\begin{proof}
  Suppose $\lambda/\mu$ is $(k,c)$-legal. Then it corresponds to some
  $k \times k$ minor
\begin{equation}\label{eq:minor}
\begin{vmatrix}
h_{i+t} & \cdots & h_{i+t+r} \\
\vdots & \ddots & \vdots \\
h_i & \cdots & h_{i+r}
\end{vmatrix}
\end{equation}
of $A_c$, where the rows and columns represented by dots are not
necessarily adjacent in~$A$. Recall that
$$
A_c = \begin{pmatrix}
h_{n-c+1} & \cdots & h_n \\
\vdots & \ddots & \vdots \\
h_1 & \cdots & h_c
\end{pmatrix}.
$$
This implies that
\begin{eqnarray*}
1 \le &i& \le n-k-c+2, \\
k-1 \le &r& \le c-1, \\
k-1 \le &t& \le n-c, \\
k \le &i+t& \le n-c+1, \\
k \le &i+r& \le c, \\
2k-1 \le &i+t+r& \le n.
\end{eqnarray*}
From Eq.~\eqref{eq:schur}, we find that
\begin{eqnarray*}
\lambda &=& (i+t+r-k+1,\ldots,i+r) \\
\mu &=& (r-k+1, \ldots, 0).
\end{eqnarray*}
Therefore
\begin{eqnarray*}
\lambda_1 &=& i+t+r-k+1 \le n-k+1, \\
\lambda_k &=& i+r \ge k, \\
\mu_1 &=& r-k+1\le c-k, \\
\lambda_1 - \lambda_k &=& t-k+1 \le n-c-k+1.
\end{eqnarray*}
The remaining inequality, $\lambda_i - \mu_i \ge k$, follows from the
fact that the diagonal terms of the submatrix must be at least $k$,
and are equal to the $\lambda_i - \mu_i$.

Now suppose $\lambda/\mu$ satisfies the above inequalities. It is
straightforward to verify that the $k \times k$ minor with entries
$$
\left|
\begin{array}{ccccc}
h_{\lambda_1 - \mu_1} & * & \cdots & * & h_{\lambda_1+k-1} \\
* & h_{\lambda_2 - \mu_2} & \cdots & * & h_{\lambda_2+k-2} \\
 & & \ddots & & \vdots \\
 * & * & \cdots & h_{\lambda_{k-1} - \mu_{k-1}} & h_{\lambda_{k-1}+1}
 \\
 * & * & \cdots & * & h_{\lambda_k}
\end{array}
\right|
$$
(where the $*$ denote entries that are determined by the choice of
rows and columns) corresponds to $\lambda/\mu$ and can be found in
$A_c$.
\end{proof}

\begin{cor}\label{c:nonskew}
  The $(k,k)$-legal diagrams are precisely non-skew $\lambda$ such
  that $\lambda_1 \le n-k+1$ and $\lambda_k \ge k$.
\end{cor}

Before we begin proving Lemma~\ref{l:toep}, we make the following
handy definition:

\begin{defn}
If $\lambda$ has $k$ parts, then the \emph{spread} of
$\lambda$ is the integer partition
$$
(\lambda_1 - \lambda_k, \lambda_2 - \lambda_k, \ldots, \lambda_{k-1} -
\lambda_k).
$$
We put a lexicographical well-ordering on spreads. That is, if $\delta
= (\delta_1,\ldots,\delta_{k-1})$ and $\epsilon =
(\epsilon_1,\ldots,\epsilon_{k-1})$ are spreads, then $\delta <
\epsilon$ if and only if the leftmost nonzero entry of $\epsilon -
\delta$ is positive.
\end{defn}

With this definition in mind, we prove our main lemma, which can be
seen as an ``inverse'' Littlewood-Richardson rule:
\begin{lemma}\label{l:main}
  Any $(k,k)$-legal diagram is a linear combination of $(k,c)$-legal
  diagrams.
\begin{proof}
  Suppose $\nu$ is a $(k,k)$-legal diagram. If $\nu_1 - \nu_k \le
  n-c-k+1$, then by Proposition~\ref{p:legal}, $\nu$ is
  $(k,c)$-legal. If $\nu_1 - \nu_k > n-c-k+1$, then its spread exceeds
  $(n-c-k+1,n-c-k+1,\ldots,n-c-k+1)$. We construct the following skew
  diagram $\lambda/\mu$ (see Figure~\ref{fig:alg}):
\begin{figure}[htbp]
\begin{tikzpicture}[>= stealth]
\young{{11,9,8,7,5}}{2}{0}
\draw[thick, ->] (6,-1) -- (7,-1);
\draw (7.25,-1) node {};
\end{tikzpicture}
\begin{tikzpicture}
\skewyoung{{15,13,12,11,9}}{{4,4,3,2}}{0}
\young{{4,2,1}}{1}{11}

\foreach \a in {4,...,10}
  \draw (\a/2 + 0.25, -0.25) node {$1$};

\foreach \a in {4,...,10}
  \draw (\a/2 + 0.25, -0.75) node {$2$};

\foreach \a in {4,...,10}
  \draw (\a/2 + 0.25, -1.25) node {$3$};

\foreach \a in {4,...,10}
  \draw (\a/2 + 0.25, -1.75) node {$4$};

\foreach \a in {4,...,8}
  \draw (\a/2 + 0.25, -2.25) node {$5$};

\draw[decorate,decoration={brace},thick] (0,0.25) to (2,0.25);
\draw (1,0.75) node {$\leq c - k$};

\draw[thick, dotted] (5.5,-2.5) -- (5.5,1);
\draw[thick, dotted] (2, -2.5) -- (2, 1);
\end{tikzpicture}
\caption{An example of constructing a $\lambda/\mu$ from a given
  non-skew $\nu$.}
\label{fig:alg}
\end{figure}
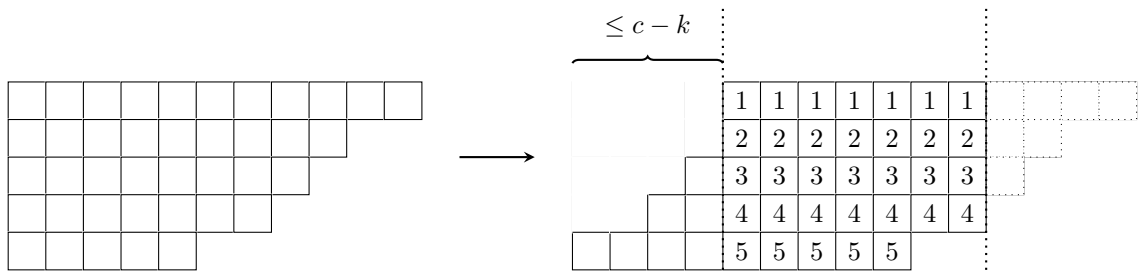

Remove the portion of the spread of $\nu$ that exceeds
$(n-c-k+1,\ldots,n-c-k+1)$, rotate it by $180^\circ$, and append it to
the bottom left of $\nu$. Algebraically, we are defining $\lambda/\mu$
by $\mu_i = \min\{\nu_1 - \nu_i, \nu_1 - \nu_k+c+k-n-1\}$ and
$\lambda_i - \mu_i = \min\{ \nu_k+n-c-k+1,\nu_i\}$. Suppose we label
$\lambda/\mu$ and consider the non-skew diagram $\pi$ that arises from
it. We can maximize the spread of $\pi$ by first maximizing the number
of $1$s in our labeling, then maximizing the number of $2$s, and
so~on. This labeling is achieved by numbering the top of each column
a~$1$, the second square of each column a~$2$, and so on. But this
labeling gives rise to $\nu$, by construction of $\lambda/\mu$. Any
other $\pi$ has spread strictly less than the spread of
$\nu$. Therefore, we have written $\nu$ as a linear combination of a
$(k,c)$-legal diagram and $(k,k)$-legal diagrams with strictly smaller
spread. The result follows by~induction.
\end{proof}
\end{lemma}

Finally, we can prove Lemma~\ref{l:toep}:

\begin{proof}[Proof of Lemma~\ref{l:toep}]
  Let $1 \le k \le c \le \frac n2$. We first show that any
  $(k,c)$-legal diagram is a linear combination of $(k,k)$-legal
  diagrams. Suppose $\lambda/\mu$ is $(k,c)$-legal. By
  Proposition~\ref{p:legal}, $\lambda_1 \le n-k+1$ and $\lambda_k \ge
  k$. If $\lambda/\mu$ is non-skew, i.e. $\mu = 0$, then by
  Corollary~\ref{c:nonskew}, $\lambda/\mu$ is $(k,k)$-legal. If
  $\lambda/\mu$ is skew, then we use the Littlewood-Richardson rule to
  write it as a sum of non-skew diagrams. Any labeling of
  $\lambda/\mu$ must have at most $n-k+1$ copies of $1$, since there
  are at most $n-k+1$ columns and each column can have at most
  one~$1$. Thus $\nu_1 \le n-k+1$ for all $\nu$ arising from a
  labeling. Similarly, a labeling must have at least $\lambda_k -
  \mu_1$ copies of $k$, since this is equal to the number of columns
  with length~$k$.  But if $\lambda/\mu$ is associated to the minor in
  Eq.~\eqref{eq:minor}, then $\lambda_k - \mu_1 = i+r - (r-k+1) =
  i+k-1 \ge k$. Therefore $\nu_k \ge k$ for all $\nu$ arising from a
  labeling. By Corollary~\ref{c:nonskew}, $\lambda/\mu$ is a sum of
  $(k,k)$-legal~diagrams.

  The other inclusion follows from Lemma~\ref{l:main}.
\end{proof}

%%%%%%%%%%%%%%%%%%%%%%%%%%%%%%%%%%%%%%%%%%%
\section{Proof of Theorem~\ref{t:pp}}%%%%%%%%%%%%%
%%%%%%%%%%%%%%%%%%%%%%%%%%%%%%%%%%%%%%%%%%%

A \emph{plane partition} of a positive integer $n$ is a finite
two-dimensional array of positive integers, weakly decreasing from
left to right and from top to bottom, with sum~$n$. For example, a
plane partition of $43$ could be
$$
\begin{array}{ccccc}
5 & 4 & 4 & 3 & 2 \\
4 & 4 & 3 & 3 & 2 \\
2 & 2 & 1 & 1 &  \\
1 & 1 &  & & \\
1 & & & & \\
\end{array}
$$
For more information on plane partitions, see \cite{Bressoud} for a
more comprehensive treatment.

We say that a plane partition \emph{fits inside an $a \times b \times
  c$ box} if there are at most $a$ rows, at most $b$ columns, and each
entry is at most~$c$.  For instance, the example above fits in an $a
\times b \times c$ box for $a,b,c = 5$. This terminology stems from
the fact that there is a natural way to view a plane partition as
cubes stacked against the corner of the first octant in $\RR^3$.

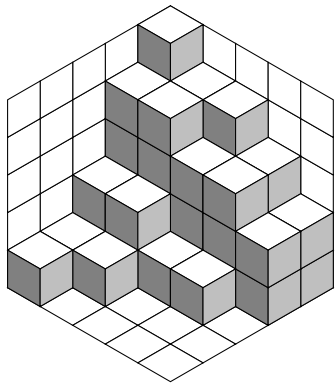
\begin{figure}[htbp]
\begin{tikzpicture}[scale=0.5]
\outline{4}{4}{4}{0};
\pp{{5,4,4,3,2}, {4,4,3,3,2}, {2,2,1,1}, {1,1}, {1}}{0};
\end{tikzpicture}
\caption{The plane partition
  $\{(5,4,4,3,2),(4,4,3,3,2),(2,2,1,1),(1,1),(1)\}$ in a $5 \times 5
  \times 5$ box}
\label{fig:pp}
\end{figure}

Let $\textsc{PP}(a,b,c)$ denote the number of plane partitions that
  fit in an $a \times b \times c$ box.  MacMahon found the following
elegant formula:
\begin{equation}\label{eq:PP}
\textsc{PP}(a,b,c) = \prod_{i=1}^a \prod_{j=1}^b \prod_{k=1}^c
\frac{i+j+k-1}{i+j+k-2}.
\end{equation}
A proof of this formula can be found in \cite{Bressoud}.  Li and
Zanello essentially prove Theorem~\ref{t:pp} in \cite{LiZanello} by
directly verifying that $|\det U_r|$ equals the right hand side
of~\eqref{eq:PP}. In this section, we give a bijective proof of this
result.  Recall the statement of Theorem~\ref{t:pp}:

\begin{fthm}
  Expressed in the monomial $\ZZ$-basis for $R =
  \ZZ[x,y,z]/(x^{a+b},y^{a+c},z^{b+c})$ the map $U_m : R_m \to
  R_{m+1}$ has its determinant $\det(U_m)$ equal, up to sign, to its
  permanent $\perm(U_m)$, and each nonzero term in its permanent
  corresponds naturally to a plane partition in an $a \times b \times
  c$ box.
\end{fthm}

\begin{proof}
  Our proof consists of two steps. First, we show that the number of
  plane partitions in an $a \times b \times c$ box is equal to the
  \emph{permanent} of $U_r$. For ease of notation, let $U = U_r$. Let
  $B_s$ denote the monomial basis of $R_s$ as before.  Recall that the
  permanent of $U$ equals
$$
\sum_\varphi \prod_{\lambda \in B_r} U_{\lambda,\varphi(\lambda)}
$$
where the sum is over all bijections $\varphi : B_r \to B_{r+1}$ and
$U_{\lambda,\mu}$ denotes the entry in row $\lambda$, column $\mu$ of
the matrix representing $U$.  Viewing the elements of $B_r \cup
B_{r+1}$ as vertices of a bipartite graph, with $B_r$ as one part and
$B_{r+1}$ as the other, and each $1$ in the matrix of $U$
corresponding to an edge, we see that each such $\varphi$ gives rise
to a perfect matching if $U_{\lambda,\varphi(\lambda)} = 1$ for all $i
\in B_r$. Since $U$ is a $0$-$1$ matrix, the permanent of $U$ counts
the number of perfect matchings in this bipartite graph.\\

\noindent\emph{Step 1: Perfect matchings to plane partitions}. Draw
an equilateral triangle with side length $a+b+c$, and insert the
monomials of $\kk[x,y,z]$ of degree $r+1 = a+b+c-1$;
Figure~\ref{fig:monomials} illustrates the case of $r=3$.
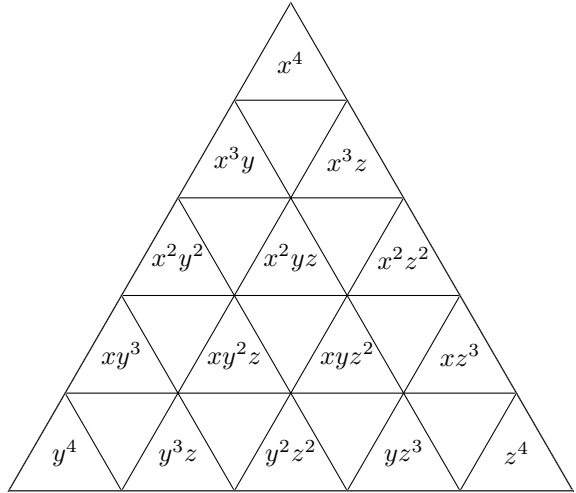
\begin{figure}[htbp]
\begin{tikzpicture}
  \draw (0,0) -- ++(60:1.5) -- ++(0:6) -- ++(-60:1.5) -- cycle;
  \draw (0,0) -- ++(60:3) -- ++(0:4.5) -- ++(-60:3) -- cycle;
  \draw (0,0) -- ++(60:4.5) -- ++(0:3) -- ++(-60:4.5) -- cycle;
  \draw (0,0) -- ++(60:6) -- ++(0:1.5) -- ++(-60:6) -- cycle;
  \draw (0,0) -- ++(60:7.5) -- ++(0:0) -- ++(-60:7.5) -- cycle;

  \draw (0,0) -- ++(60:1.5) -- ++(-60:1.5) -- cycle;
  \draw (0,0) -- ++(60:3) -- ++(-60:3) -- cycle;
  \draw (0,0) -- ++(60:4.5) -- ++(-60:4.5) -- cycle;
  \draw (0,0) -- ++(60:6) -- ++(-60:6) -- cycle;

  \draw (7.5,0) -- ++(120:1.5) -- ++(240:1.5) -- cycle;
  \draw (7.5,0) -- ++(120:3) -- ++(240:3) -- cycle;
  \draw (7.5,0) -- ++(120:4.5) -- ++(240:4.5) -- cycle;
  \draw (7.5,0) -- ++(120:6) -- ++(240:6) -- cycle;

  \draw[thick,white] (0,0) -- ++(60:7.5);
  \draw (0,0) -- ++(60:7.5);
  \draw[thick,white] (7.5,0) -- ++(120:7.5);
  \draw (7.5,0) -- ++(120:7.5);
  \draw[thick,white] (0,0) -- (7.5,0);
  \draw (0,0) -- (7.5,0);

  \draw (0.75, 0.5) node (040) {$y^4$}
    ++(1.5, 0) node (031) {$y^3z$}
    ++(120:1.5) node (130) {$xy^3$}
    ++(60:1.5) node (220) {$x^2y^2$}
    ++(60:1.5) node (310) {$x^3y$}
    ++(60:1.5) node (400) {$x^4$}
    ++(-60:1.5) node (301) {$x^3z$}
    ++(-60:1.5) node (202) {$x^2z^2$}
    ++(-60:1.5) node (103) {$xz^3$}
    ++(-60:1.5) node (004) {$z^4$}
    ++(-1.5,0) node (013) {$yz^3$}
    ++(120:1.5) node (112) {$xyz^2$}
    ++(120:1.5) node (211) {$x^2yz$}
    ++(240:1.5) node (121) {$xy^2z$}
    ++(-60:1.5) node (022) {$y^2z^2$};

  \draw (8,4) node {};
\end{tikzpicture}
\caption{Monomials in $B_{4}$}
\label{fig:monomials}
\end{figure}

Fill the remaining spaces with the monomials of degree $r = a+b+c-2$
as in Figure~\ref{fig:monomials2}; notice that two monomials in the
diagram are adjacent if and only if one divides the other.

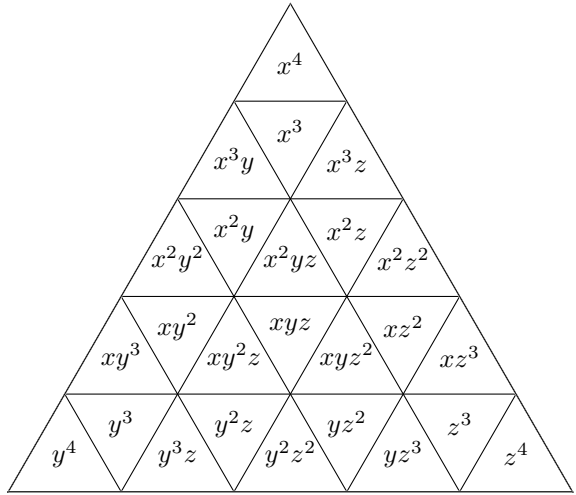
\begin{figure}[htbp]
\begin{tikzpicture}
  \draw (0,0) -- ++(60:1.5) -- ++(0:6) -- ++(-60:1.5) -- cycle;
  \draw (0,0) -- ++(60:3) -- ++(0:4.5) -- ++(-60:3) -- cycle;
  \draw (0,0) -- ++(60:4.5) -- ++(0:3) -- ++(-60:4.5) -- cycle;
  \draw (0,0) -- ++(60:6) -- ++(0:1.5) -- ++(-60:6) -- cycle;
  \draw (0,0) -- ++(60:7.5) -- ++(0:0) -- ++(-60:7.5) -- cycle;

  \draw (0,0) -- ++(60:1.5) -- ++(-60:1.5) -- cycle;
  \draw (0,0) -- ++(60:3) -- ++(-60:3) -- cycle;
  \draw (0,0) -- ++(60:4.5) -- ++(-60:4.5) -- cycle;
  \draw (0,0) -- ++(60:6) -- ++(-60:6) -- cycle;

  \draw (7.5,0) -- ++(120:1.5) -- ++(240:1.5) -- cycle;
  \draw (7.5,0) -- ++(120:3) -- ++(240:3) -- cycle;
  \draw (7.5,0) -- ++(120:4.5) -- ++(240:4.5) -- cycle;
  \draw (7.5,0) -- ++(120:6) -- ++(240:6) -- cycle;

  \draw[thick,white] (0,0) -- ++(60:7.5);
  \draw (0,0) -- ++(60:7.5);
  \draw[thick,white] (7.5,0) -- ++(120:7.5);
  \draw (7.5,0) -- ++(120:7.5);
  \draw[thick,white] (0,0) -- (7.5,0);
  \draw (0,0) -- (7.5,0);

  \draw (0.75, 0.5) node (040) {$y^4$}
    ++(1.5, 0) node (031) {$y^3z$}
    ++(120:1.5) node (130) {$xy^3$}
    ++(60:1.5) node (220) {$x^2y^2$}
    ++(60:1.5) node (310) {$x^3y$}
    ++(60:1.5) node (400) {$x^4$}
    ++(-60:1.5) node (301) {$x^3z$}
    ++(-60:1.5) node (202) {$x^2z^2$}
    ++(-60:1.5) node (103) {$xz^3$}
    ++(-60:1.5) node (004) {$z^4$}
    ++(-1.5,0) node (013) {$yz^3$}
    ++(120:1.5) node (112) {$xyz^2$}
    ++(120:1.5) node (211) {$x^2yz$}
    ++(240:1.5) node (121) {$xy^2z$}
    ++(-60:1.5) node (022) {$y^2z^2$};

  \draw (1.5, 0.9) node (030) {$y^3$}
    ++(60:1.5) node (120) {$xy^2$}
    ++(60:1.5) node (210) {$x^2y$}
    ++(60:1.5) node (300) {$x^3$}
    ++(-60:1.5) node (201) {$x^2z$}
    ++(-60:1.5) node (102) {$xz^2$}
    ++(-60:1.5) node (003) {$z^3$}
    ++(-1.5,0) node (012) {$yz^2$}
    ++(-1.5,0) node (021) {$y^2z$}
    ++(60:1.5) node (111) {$xyz$};

\end{tikzpicture}
\caption{Adding monomials in $B_3$}
\label{fig:monomials2}
\end{figure}
Now truncate the equilateral triangle by removing a triangle
of length $c$ from the top corner, a triangle of length $b$ from the
lower left corner, and a triangle of length $a$ from the lower right
corner. We are left with an $a \times b \times c$ hexagon containing
the monomials in $B_r \cup B_{r+1}$, the vertices of the bipartite
graph mentioned in the previous paragraph. Figure~\ref{fig:monomials3}
shows the hexagon for $(a,b,c) = (2,2,1)$. Figure~\ref{fig:matching}
shows this hexagons with monomials in $B_r$ as black dots and
monomials in $B_{r+1}$ as white dots.
\begin{figure}[htbp]
\begin{tikzpicture}
  \draw (0,0) -- ++(60:1.5) -- ++(0:6) -- ++(-60:1.5) -- cycle;
  \draw (0,0) -- ++(60:3) -- ++(0:4.5) -- ++(-60:3) -- cycle;
  \draw (0,0) -- ++(60:4.5) -- ++(0:3) -- ++(-60:4.5) -- cycle;
  \draw (0,0) -- ++(60:6) -- ++(0:1.5) -- ++(-60:6) -- cycle;
  \draw (0,0) -- ++(60:7.5) -- ++(0:0) -- ++(-60:7.5) -- cycle;

  \draw (0,0) -- ++(60:1.5) -- ++(-60:1.5) -- cycle;
  \draw (0,0) -- ++(60:3) -- ++(-60:3) -- cycle;
  \draw (0,0) -- ++(60:4.5) -- ++(-60:4.5) -- cycle;
  \draw (0,0) -- ++(60:6) -- ++(-60:6) -- cycle;

  \draw (7.5,0) -- ++(120:1.5) -- ++(240:1.5) -- cycle;
  \draw (7.5,0) -- ++(120:3) -- ++(240:3) -- cycle;
  \draw (7.5,0) -- ++(120:4.5) -- ++(240:4.5) -- cycle;
  \draw (7.5,0) -- ++(120:6) -- ++(240:6) -- cycle;

  \draw[thick,white] (0,0) -- ++(60:7.5);
  \draw (0,0) -- ++(60:7.5);
  \draw[thick,white] (7.5,0) -- ++(120:7.5);
  \draw (7.5,0) -- ++(120:7.5);
  \draw[thick,white] (0,0) -- (7.5,0);
  \draw (0,0) -- (7.5,0);

  \draw (0.75, 0.5) node (040) {$y^4$}
    ++(1.5, 0) node (031) {$y^3z$}
    ++(120:1.5) node (130) {$xy^3$}
    ++(60:1.5) node (220) {$x^2y^2$}
    ++(60:1.5) node (310) {$x^3y$}
    ++(60:1.5) node (400) {$x^4$}
    ++(-60:1.5) node (301) {$x^3z$}
    ++(-60:1.5) node (202) {$x^2z^2$}
    ++(-60:1.5) node (103) {$xz^3$}
    ++(-60:1.5) node (004) {$z^4$}
    ++(-1.5,0) node (013) {$yz^3$}
    ++(120:1.5) node (112) {$xyz^2$}
    ++(120:1.5) node (211) {$x^2yz$}
    ++(240:1.5) node (121) {$xy^2z$}
    ++(-60:1.5) node (022) {$y^2z^2$};

  \draw (1.5, 0.9) node (030) {$y^3$}
    ++(60:1.5) node (120) {$xy^2$}
    ++(60:1.5) node (210) {$x^2y$}
    ++(60:1.5) node (300) {$x^3$}
    ++(-60:1.5) node (201) {$x^2z$}
    ++(-60:1.5) node (102) {$xz^2$}
    ++(-60:1.5) node (003) {$z^3$}
    ++(-1.5,0) node (012) {$yz^2$}
    ++(-1.5,0) node (021) {$y^2z$}
    ++(60:1.5) node (111) {$xyz$};

  \fill[draw,white] (0,0) -- ++(3,0) -- ++(120:3) -- cycle;
  \fill[draw,white] ++(60:7.5) -- ++(-60:1.5) -- ++(-1.5,0) -- cycle;
  \fill[draw,white] (7.5,0) -- ++(120:3) -- ++(240:3) -- cycle;
  \draw ++(3,0) -- ++(120:3);
  \draw ++(4.5,0) -- ++(60:3);
  \draw ++(60:6) -- ++(1.5,0);

\end{tikzpicture}
\caption{Truncating equilateral triangle}
\label{fig:monomials3}
\end{figure}
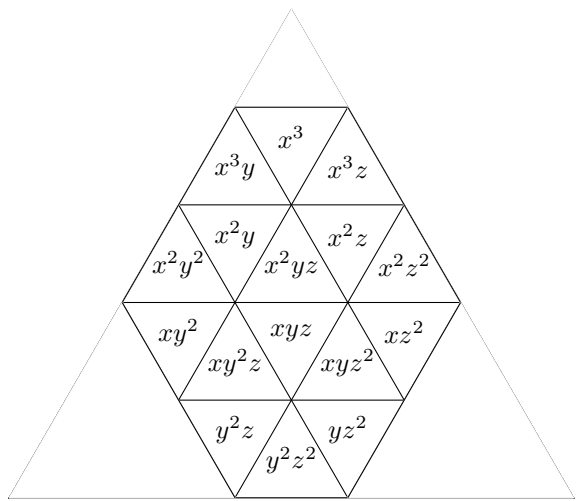
A perfect matching between $B_r$ and $B_{r+1}$ is a bijection such
that for any $\lambda \in B_r$, $\lambda$ and its image are adjacent
in the hexagon, i.e., they lie in adjacent equilateral triangles. We
can thus represent a perfect matching pictorially as in
Figure~\ref{fig:matching}. It is thus evident that there is a
bijection between perfect matchings and rhombus tilings of this $a
\times b \times c$ hexagon.  But the bijection between such rhombus
tilings and plane partitions in an $a \times b \times c$ box is
well-known; we view the tiling as a three-dimensional picture of the
plane partition (rotated by $90^\circ$). Hence, the number of perfect
matchings is equal to the number of such plane partitions, as desired.
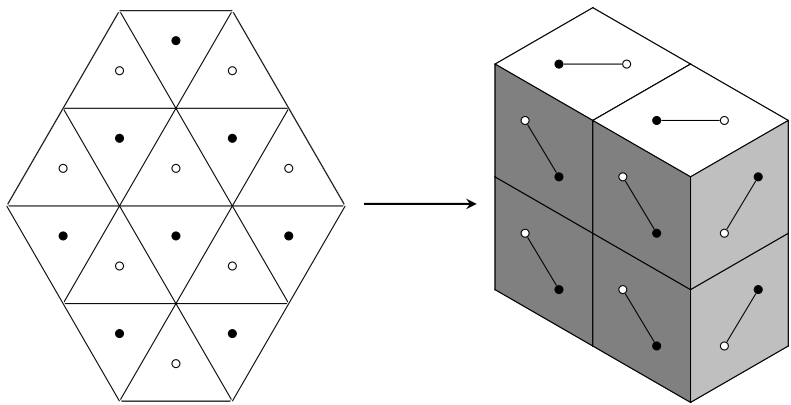
\begin{figure}[htbp]
\hspace{-60pt}
\begin{tikzpicture}[>=stealth]
  \draw (0,0) -- ++(60:1.5) -- ++(0:6) -- ++(-60:1.5) -- cycle;
  \draw (0,0) -- ++(60:3) -- ++(0:4.5) -- ++(-60:3) -- cycle;
  \draw (0,0) -- ++(60:4.5) -- ++(0:3) -- ++(-60:4.5) -- cycle;
  \draw (0,0) -- ++(60:6) -- ++(0:1.5) -- ++(-60:6) -- cycle;
  \draw (0,0) -- ++(60:7.5) -- ++(0:0) -- ++(-60:7.5) -- cycle;

  \draw (0,0) -- ++(60:1.5) -- ++(-60:1.5) -- cycle;
  \draw (0,0) -- ++(60:3) -- ++(-60:3) -- cycle;
  \draw (0,0) -- ++(60:4.5) -- ++(-60:4.5) -- cycle;
  \draw (0,0) -- ++(60:6) -- ++(-60:6) -- cycle;

  \draw (7.5,0) -- ++(120:1.5) -- ++(240:1.5) -- cycle;
  \draw (7.5,0) -- ++(120:3) -- ++(240:3) -- cycle;
  \draw (7.5,0) -- ++(120:4.5) -- ++(240:4.5) -- cycle;
  \draw (7.5,0) -- ++(120:6) -- ++(240:6) -- cycle;

  \draw[thick,white] (0,0) -- ++(60:7.5);
  \draw (0,0) -- ++(60:7.5);
  \draw[thick,white] (7.5,0) -- ++(120:7.5);
  \draw (7.5,0) -- ++(120:7.5);
  \draw[thick,white] (0,0) -- (7.5,0);
  \draw (0,0) -- (7.5,0);

  \draw (0.75, 0.5) node (040) [circle, fill=white] {}
    ++(1.5, 0) node (031) [circle, fill=white] {}
    ++(120:1.5) node (130) [circle, fill=white] {}
    ++(60:1.5) node (220) [circle, fill=white] {}
    ++(60:1.5) node (310) [circle, fill=white] {}
    ++(60:1.5) node (400) [circle, fill=white] {}
    ++(-60:1.5) node (301) [circle, fill=white] {}
    ++(-60:1.5) node (202) [circle, fill=white] {}
    ++(-60:1.5) node (103) [circle, fill=white] {}
    ++(-60:1.5) node (004) [circle, fill=white] {}
    ++(-1.5,0) node (013) [circle, fill=white] {}
    ++(120:1.5) node (112) [circle, fill=white] {}
    ++(120:1.5) node (211) [circle, fill=white] {}
    ++(240:1.5) node (121) [circle, fill=white] {}
    ++(-60:1.5) node (022) [circle, fill=white] {};

  \draw (1.5, 0.9) node (030) [circle, fill=black] {}
    ++(60:1.5) node (120) [circle, fill=black] {}
    ++(60:1.5) node (210) [circle, fill=black] {}
    ++(60:1.5) node (300) [circle, fill=black] {}
    ++(-60:1.5) node (201) [circle, fill=black] {}
    ++(-60:1.5) node (102) [circle, fill=black] {}
    ++(-60:1.5) node (003) [circle, fill=black] {}
    ++(-1.5,0) node (012) [circle, fill=black] {}
    ++(-1.5,0) node (021) [circle, fill=black] {}
    ++(60:1.5) node (111) [circle, fill=black] {};

  \fill[draw,very thick,white] (0,0) -- ++(3,0) -- ++(120:3) -- cycle;
  \fill[draw,very thick,white] ++(60:7.5) -- ++(-60:1.5) -- ++(-1.5,0)
  -- cycle;
  \fill[draw,very thick,white] (7.5,0) -- ++(120:3) -- ++(240:3) --
  cycle;
  \draw ++(3,0) -- ++(120:3);
  \draw ++(4.5,0) -- ++(60:3);
  \draw ++(60:6) -- ++(1.5,0);

  \draw [thick, ->] (6.25,2.625) -- (7.75,2.625);
  \draw (7.75, 2.625) node {};
\end{tikzpicture}
\begin{tikzpicture}[scale=1.5]
  \tikzstyle{every circle node}=[draw,minimum size=2pt,inner sep=0]
  \outline{1}{0}{1}{1}
  \pp{{2,2}}{1}
\end{tikzpicture}
\caption{Left: Simplified diagram. Right: A perfect matching viewed as
  a plane partition}
\label{fig:matching}
\end{figure}
\\ 

\noindent\emph{Step 2: Permanent equals determinant}. Consider the
bijection between perfect matchings and plane partitions given
above. To get from one plane partition to another, we perform the
operations of adding or removing a block. For example,
Figure~\ref{fig:rotation} shows the plane partition with the
right-most block removed (or added, depending on perspective, but we
fix perspective at the beginning).
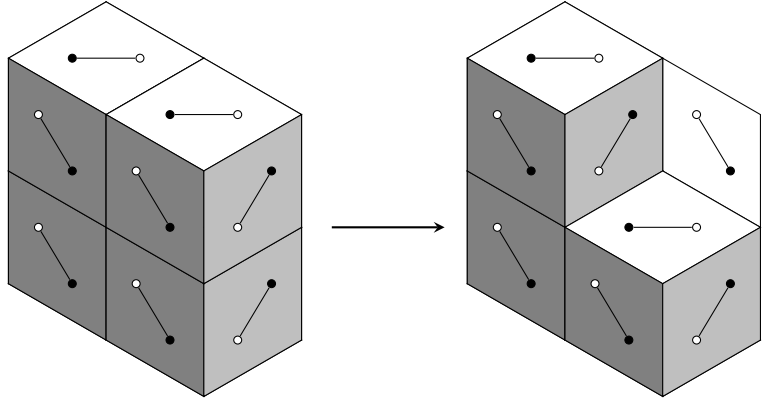
\begin{figure}[htbp]
\begin{tikzpicture}[scale=1.5, >=stealth]
  \tikzstyle{every circle node}=[draw,minimum size=2pt,inner sep=0]
  \outline{1}{0}{1}{1}
  \pp{{2,2}}{1}
  \draw [thick, ->] (2,0) -- (3,0);
  \draw (3,0.5) node {};
\end{tikzpicture}
\begin{tikzpicture}[scale=1.5]
  \tikzstyle{every circle node}=[draw,minimum size=2pt,inner sep=0]
  \outline{1}{0}{1}{1}
  \pp{{2,1}}{1}
\end{tikzpicture}
\caption{Removing a block}
\label{fig:rotation}
\end{figure}

As Figure~\ref{fig:rotation} shows, if $\varphi$ is the perfect
matching associated with a plane partition, the associated perfect
matching after such an operation is $\sigma\varphi$, where $\sigma$ is
a 3-cycle. Note that $\sigma$ is an even permutation; hence, we can
get from any perfect matching to another using even permutations. This
shows that all perfect matchings have the same sign, completing the
proof.
\end{proof}

%%%%%%%%%%%%%%%%%%%%%%%%%%%%%%%%%%%%%%%%%%%%%%%
\section{Proofs of Theorems~\ref{t:cspp}, \ref{t:tcpp}, and \ref{t:cstcpp}}%
%%%%%%%%%%%%%%%%%%%%%%%%%%%%%%%%%%%%%%%%%%%%%%%
In \cite{Bressoud}, Bressoud describes the ten symmetry classes of
plane partitions that are of interest. Each class consists of plane
partitions that are invariant under certain actions. Each of these
actions is some composition $\rho^i \tau^j \kappa^k$ where $(i,j,k)
\in \{0,1\}^3$, $\rho$ is cyclically permuting the three coordinate
axes, $\tau$ is reflection across the $y = z$ plane, and $\kappa$ is
complementation. Viewing plane partitions as rhombus tilings in the
plane, it is straightforward to check that $\rho$ acts by
counterclockwise rotation by $120^\circ$, $\tau$ acts by reflection
about the vertical axis of symmetry, and $\kappa$ acts by rotation
by~$180^\circ$. Note that $\<\rho,\tau,\kappa\> = D_{12}$, the
dihedral group of order~$12$, which is the group of symmetries on a
regular hexagon. Given the above bijection, each symmetry class of
plane partitions can be viewed as a subset of perfect matchings of the
aforementioned bipartite graphs which are invariant under some
subgroup of $D_{12}$. The ten symmetry classes are enumerated in
Figure~\ref{fig:PPtable} in the same order as in~\cite{Bressoud}.
\begin{figure}[htbp]
\begin{tabular}{|c|l|c|c|}
\hline
Number & Class & Abbreviation &  Subgroup \\
\hline
1 & Plane partitions & \textsc{PP} & $\<e\>$ \\
\hline
2 & Symmetric plane partitions & \textsc{SPP} & $\<\tau\>$ \\
\hline
3 & Cyclically symmetric plane partitions & \textsc{CSPP} & $\<\rho\>$
\\
\hline
4 & Totally symmetric plane partitions & \textsc{TSPP} &
$\<\rho,\tau\>$ \\
\hline
5 & Self-complementary plane partitions & \textsc{SCPP} & $\<\kappa\>$
\\
\hline
6 & Transpose complement plane partitions & \textsc{TCPP} &
$\<\tau\kappa\>$ \\
\hline
7 & Symmetric self-complementary plane partitions & \textsc{SSCPP} &
$\<\tau,\kappa\>$\\
\hline
8 & Cyclically symmetric transpose complement plane partitions
& \textsc{CSTCPP} & $\<\rho,\tau\kappa\>$ \\
\hline
9 & Cyclically symmetric self-complementary plane partitions & \textsc{CSSCPP} &
$\<\rho,\kappa\>$ \\
\hline
10 & Totally symmetric self-complementary plane partitions & \textsc{TSSCPP} &
$\<\rho,\tau,\kappa\> = D_{12}$ \\
\hline
\end{tabular}
\caption{The symmetry classes of plane partitions}
\label{fig:PPtable}
\end{figure}

We have shown that, for $R = \kk[x,y,z]/(x^{a+b},y^{a+c},z^{b+c})$ and
$r = a+b+c-2$, we have $|\det U_r | = \textsc{PP}(a,b,c)$. In this
section, we show that, for certain natural submodules of $R$ over
various invariant subrings (for appropriate group actions in each
case), the determinant of the restriction of $U_r$ to the submodule is
equal (up to sign) to the number of a symmetry class of plane
partitions that fit in an $a \times b \times c$ box.

\newpage

%%%%%%%%%%%%%%%%%%%%%%%%%%%%%%%%%%%%%%%%%%%%%%
\subsection*{Proof of Theorem~\ref{t:cspp}}%%%%%%%%%%%%%
%%%%%%%%%%%%%%%%%%%%%%%%%%%%%%%%%%%%%%%%%%%%%%
Recall the statement of Theorem~\ref{t:cspp}:

\begin{fthm}
  Assuming $a=b=c$, let $C_3 = \ZZ/3\ZZ = \{1,\rho,\rho^2\}$ act on $R
  = \ZZ[x,y,z]/(x^{2a},y^{2a},z^{2a})$ by cycling the variables $x
  \stackrel{\rho}{\to} y \stackrel{\rho}{\to} z \stackrel{\rho}{\to}
  x$. Then the map $U_m|_{R^{C_3}}$ restricted to the $m$-th
  homogeneous component of the $C_3$-invariant subring $R^{C_3}$ has
  $\det(U_m|_{R^{C_3}})$ equal, up to sign, to the number of
  cyclically symmetric plane partitions in an $a \times a \times a$
  box.
\end{fthm}

\begin{proof}
  The proof is in the same spirit as our proof for vanilla plane
  partitions: we show that the permanent of $U'_r$ is equal to the
  cyclically invariant perfect matchings (which correspond to
  cyclically symmetric plane partitions), and then show that the
  permanent equals the determinant up to sign.

  The elements $x^iy^jz^k + y^iz^jx^k + z^ix^jy^k$ form a basis for
  $R^\rho$.  Hence, $U'_r$ can be realized as a matrix with columns
  indexed by the basis elements of rank $r$ of $R^\rho$, and rows
  indexed by the basis elements of rank $r+1$ of $R^\rho$.

  The element $\rho$ acts on the hexagonal bipartite graph $G$ by
  rotation by 120 degrees, so we can consider the quotient graph
  $G/\<\rho\>$, where the vertices and edges are $\<\rho\>$-orbits of
  vertices and edges in $G$, respectively.
  Figure~\ref{fig:orbitgraph} shows an example of the honeycomb graph
  and its corresponding quotient graph when $a = b = c = 2$.
 
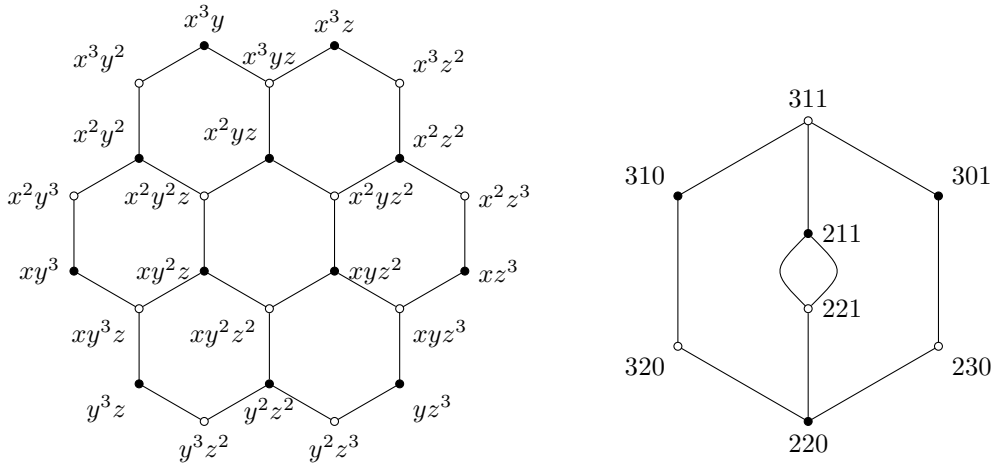
\begin{figure}[htbp]
\begin{tikzpicture}
\matrix{
  % inner guy
  \draw (0,0) node (211) [circle, fill=black,label=120:$x^2yz$] {}
    -- ++(210:1) node (221) [circle, fill=white,label=left:$x^2y^2z$] {}
    -- ++(270:1) node (121) [circle, fill=black,label=left:$xy^2z$] {}
    -- ++(-30:1) node (122) [circle, fill=white,label=210:$xy^2z^2$] {}
    -- ++(30:1) node (112) [circle, fill=black,label=right:$xyz^2$] {}
    -- ++(90:1) node (212) [circle, fill=white,label=right:$x^2yz^2$] {}
    -- (211);

  % outer guy
  \draw (211)
    -- ++(90:1) node (311) [circle, fill=white,label=90:$x^3yz$] {}
    -- ++(30:1) node (301) [circle, fill=black,label=90:$x^3z$] {}
    -- ++(-30:1) node (302) [circle, fill=white,label=30:$x^3z^2$] {}
    -- ++(-90:1) node (202) [circle, fill=black,label=30:$x^2z^2$] {}
    -- ++(-30:1) node (203) [circle, fill=white,label=right:$x^2z^3$] {}
    -- ++(-90:1) node (103) [circle, fill=black,label=right:$xz^3$] {}
    -- ++(210:1) node (113) [circle, fill=white,label=-30:$xyz^3$] {}
    -- ++(-90:1) node (013) [circle, fill=black,label=-30:$yz^3$] {}
    -- ++(210:1) node (023) [circle, fill=white,label=-90:$y^2z^3$] {}
    -- ++(150:1) node (022) [circle, fill=black,label=-90:$y^2z^2$] {}
    -- ++(210:1) node (032) [circle, fill=white,label=-90:$y^3z^2$] {}
    -- ++(150:1) node (031) [circle, fill=black,label=210:$y^3z$] {}
    -- ++(90:1) node (131) [circle, fill=white,label=210:$xy^3z$] {}
    -- ++(150:1) node (130) [circle, fill=black,label=left:$xy^3$] {}
    -- ++(90:1) node (230) [circle, fill=white,label=left:$x^2y^3$] {}
    -- ++(30:1) node (220) [circle, fill=black,label=150:$x^2y^2$] {}
    -- ++(90:1) node (320) [circle, fill=white,label=150:$x^3y^2$] {}
    -- ++(30:1) node (310) [circle, fill=black,label=90:$x^3y$] {}
    -- (311);

  % fill in some edges
  \draw (220) -- (221);
  \draw (121) -- (131);
  \draw (122) -- (022);
  \draw (112) -- (113);
  \draw (212) -- (202); &
  \draw[white] (0,0) -- (1,0); &
  \draw (0,-1) node (211) [circle, fill=black, label=right:$211$]{}
  .. controls (-0.5, -1.5) .. (0,-2) node (221) [circle, fill=white,
  label=right:$221$]{};
  \draw (211) .. controls (0.5, -1.5) .. (221);
  \draw (0,-1) -- ++(90:1.5) node (311) [circle, fill=white,
  label=90:$311$]{}
  -- ++(210:2) node (310) [circle, fill=black, label=150:$310$]{}
  -- ++(-90:2) node (320) [circle, fill=white, label=210:$320$]{}
  -- ++(-30:2) node (220) [circle, fill=black,  label=-90:$220$]{}
  -- ++(30:2) node (230) [circle, fill=white, label=-30:$230$]{}
  -- ++(90:2) node (301) [circle, fill=black,  label=30:$301$]{}
  -- (311);
  \draw (220) -- (221);
  \\
};
\end{tikzpicture}
\caption{$G$ (left) and $G/\<\rho\>$ (right) for $a = b = c = 2$}
\label{fig:orbitgraph}
\end{figure}
  Since $r = 3a-2 \equiv 1 \pmod 3$ and $r+1 \equiv 2 \pmod 3$, no
  monomial in either rank is fixed under~$\rho$.  It follows that each
  vertex and edge orbit correspond to three vertices and edges in $G$,
  respectively.  Hence, cyclically symmetric perfect matchings on $G$
  correspond to perfect matchings on $G/\<\rho\>$.

  Moreover, there is a natural correspondence between the vertices in
  the quotient graph and the basis elements of the two middle ranks.
  Depending on the context, we will abuse notation and write $[i,j,k]$
  to mean both a vertex orbit $[x^iy^jz^k]$ or a basis element
  $x^iy^jz^k + y^iz^jx^k + z^ix^jy^k$.

  In the quotient graph, every vertex orbit $[i,j,k]$ is connected to
  $[i+1,j,k]$, $[i,j+1,k]$ or $[i,j,k+1]$ (if those orbits exist), and
  similarly, $U_r'$ sends any basis element $[i,j,k]$ to a sum of
  basis elements $[i+1,j,k] + [i,j+1,k] + [i,j,k+1]$ (if those
  elements exist.)

  Note that the orbits $[a,a,a-1]$ and $[a,a-1,a-1]$ share two edges
  in the quotient graph.  The only way for two orbits to share two
  edges is if all elements of both orbits entirely make up the
  vertices of a hexagon in~$G$.  Hence, only the middle hexagon gives
  rise to this double edge in the quotient graph.  Likewise, as basis
  elements, $U_r'[a,a-1,a-1] = [a+1,a-1,a-1] + 2 \cdot [a,a,a-1]$, and
  $[a,a-1,a-1]$ is the only element that gets sent to two times
  another basis element.  Hence, $U_r'$ is the adjacency matrix of
  $G/\<\rho\>$.  The number of perfect matchings on $G/\<\rho\>$ is
  equal to the permanent of $U_r'$ (by treating the sole $2$ in $U_r'$
  as the two possible ways to have $[a,a-1,a-1]$ and $[a,a,a-1]$
  connected in a perfect matching), and it follows that the permanent
  of $U_r'$ equals the number of cyclically symmetric plane
  partitions.

  Next, we show that the permanent equals the determinant up to sign
  by showing all the matchings of $G/\<\rho\>$ have the same sign.  To
  get from a cyclically symmetric partition to one with fewer blocks,
  we can remove three blocks in the same orbit, or remove the block in
  the center of the partition.  If $\varphi$ is the perfect matching
  in $G/\<\rho\>$ that corresponds to the partition, the partition
  after removing three blocks in the same orbit corresponds to $\sigma
  \varphi$, where $\sigma$ is a $3$-cycle, so this operation doesn't
  change the sign of the matching. Removing the central block
  corresponds to switching the edge between $[a, a-1, a-1]$ and
  $[a,a,a-1]$ to the other edge, and thus does not change the
  sign. Figure~\ref{fig:centralblock} gives an example of these
  removals in action.

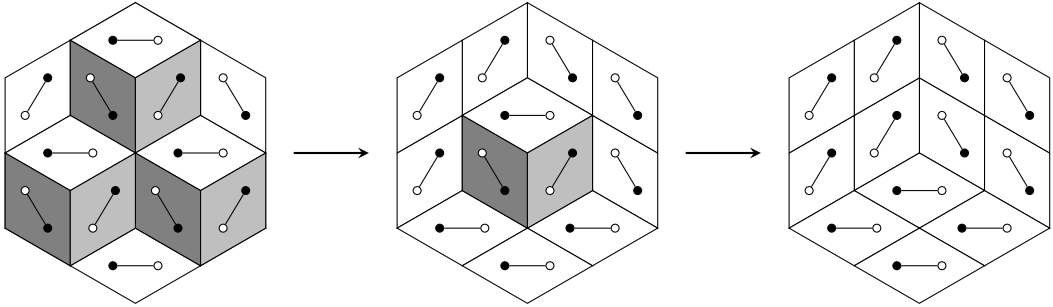
\begin{figure}[htbp]
\begin{tikzpicture}[>= stealth]
\matrix{
\outline{1}{1}{1}{1};
\pp{{2,1}, {1}}{1};&
\draw (0,0) node {};
\draw [thick, ->] (0.25,0) node{} -- (1.25,0) node{};
\draw (1.5,0) node {};  &
\outline{1}{1}{1}{1};
\pp{{1}}{1};&
\draw (0,0) node {};
\draw [thick, ->] (0.25,0) node{} -- (1.25,0) node{};
\draw (1.5,0) node {};  &
\outline{1}{1}{1}{1};\\
};
\end{tikzpicture}
\caption{Removing three blocks in the same orbit, then removing the
  central block}
\label{fig:centralblock}
\end{figure}
The signs of all the matchings in $G/\<\rho\>$ are the same, so the
permanent is equal to the determinant up to sign, and the result
follows.

\end{proof}

%%%%%%%%%%%%%%%%%%%%%%%%%%%%%%%%%%%%%%%%%%%%%%
\subsection*{Proof of Theorem~\ref{t:tcpp}}%%%%%%%%%%%%%
%%%%%%%%%%%%%%%%%%%%%%%%%%%%%%%%%%%%%%%%%%%%%%
Recall the statement of Theorem~\ref{t:tcpp}:

\begin{fthm}
  Assuming $a=b$ and the product $abc$ is even, let $C_2 = \ZZ/2\ZZ =
  \{1,\tau\kappa\}$ act on $R = \ZZ[x,y,z]/(x^{2a},y^{a+c},z^{a+c})$
  by swapping $y \leftrightarrow z$.  Then the map $U_m|_{R^{C_2,-}}$
  restricted to the $m$-th homogeneous component of the anti-invariant
  submodule $R^{C_3,-} := \{f \in R : \tau\kappa(f) = -f\}$ has
  $\det(U_m|_{R^{C_2,-}})$ equal, up to sign, to the number of
  transpose complementary plane partitions in an $a \times a \times c$
  box.
\end{fthm}

\begin{proof}
  This time, the transpose complementary plane partitions correspond
  to perfect matchings on $G$ that are symmetric about the vertical
  axis of symmetry (flip symmetric).  Note that all the vertices on
  the vertical axis are of the form $x^iy^jz^j$, and in any symmetric
  perfect matching, $x^iy^jz^j$ must be matched with $x^{i+1}y^jz^j$
  to preserve flip symmetry.

  On $G/\<\tau\kappa\>$, erase all vertices $[i,j,j]$ and all edges
  coming out of them.  Call this new graph~$G'$.  From flip symmetry
  and our previous observation, flip symmetric matchings on $G$
  correspond directly to matchings on $G'$.  Figure~\ref{fig:TSorbit}
  gives an example of $G'$ when $a = b = c = 2$.

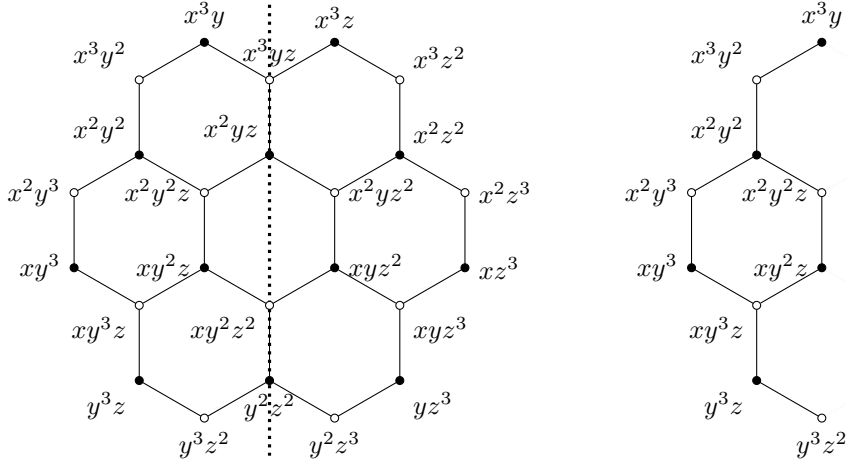
\begin{figure}
\begin{tikzpicture}
  \matrix{
  \draw [very thick, dotted] (0,2) -- (0,-4);
  % inner guy
  \draw (0,0) node (211) [circle, fill=black,label=120:$x^2yz$] {}
    -- ++(210:1) node (221) [circle, fill=white,label=left:$x^2y^2z$] {}
    -- ++(270:1) node (121) [circle, fill=black,label=left:$xy^2z$] {}
    -- ++(-30:1) node (122) [circle, fill=white,label=210:$xy^2z^2$] {}
    -- ++(30:1) node (112) [circle, fill=black,label=right:$xyz^2$] {}
    -- ++(90:1) node (212) [circle, fill=white,label=right:$x^2yz^2$] {}
    -- (211);

  % outer guy
  \draw (211)
    -- ++(90:1) node (311) [circle, fill=white,label=90:$x^3yz$] {}
    -- ++(30:1) node (301) [circle, fill=black,label=90:$x^3z$] {}
    -- ++(-30:1) node (302) [circle, fill=white,label=30:$x^3z^2$] {}
    -- ++(-90:1) node (202) [circle, fill=black,label=30:$x^2z^2$] {}
    -- ++(-30:1) node (203) [circle, fill=white,label=right:$x^2z^3$] {}
    -- ++(-90:1) node (103) [circle, fill=black,label=right:$xz^3$] {}
    -- ++(210:1) node (113) [circle, fill=white,label=-30:$xyz^3$] {}
    -- ++(-90:1) node (013) [circle, fill=black,label=-30:$yz^3$] {}
    -- ++(210:1) node (023) [circle, fill=white,label=-90:$y^2z^3$] {}
    -- ++(150:1) node (022) [circle, fill=black,label=-90:$y^2z^2$] {}
    -- ++(210:1) node (032) [circle, fill=white,label=-90:$y^3z^2$] {}
    -- ++(150:1) node (031) [circle, fill=black,label=210:$y^3z$] {}
    -- ++(90:1) node (131) [circle, fill=white,label=210:$xy^3z$] {}
    -- ++(150:1) node (130) [circle, fill=black,label=left:$xy^3$] {}
    -- ++(90:1) node (230) [circle, fill=white,label=left:$x^2y^3$] {}
    -- ++(30:1) node (220) [circle, fill=black,label=150:$x^2y^2$] {}
    -- ++(90:1) node (320) [circle, fill=white,label=150:$x^3y^2$] {}
    -- ++(30:1) node (310) [circle, fill=black,label=90:$x^3y$] {}
    -- (311);

  % fill in some edges
  \draw (220) -- (221);
  \draw (121) -- (131);
  \draw (122) -- (022);
  \draw (112) -- (113);
  \draw (212) -- (202); &
  \draw[white] (0,0) -- (1,0); &
  % inner guy
  \draw (0,0) node (211) [circle, fill=black] {}
    -- ++(210:1) node (221) [circle, fill=white,label=left:$x^2y^2z$] {}
    -- ++(270:1) node (121) [circle, fill=black,label=left:$xy^2z$] {}
    -- ++(-30:1) node (122) [circle, fill=white] {}
    -- ++(30:1) node (112) [circle, fill=black,label=right:$xyz^2$] {}
    -- ++(90:1) node (212) [circle, fill=white,label=right:$x^2yz^2$] {}
    -- (211);

  % outer guy
  \draw (211)
    -- ++(90:1) node (311) [circle, fill=white,label=90:$x^3yz$] {}
    -- ++(30:1) node (301) [circle, fill=black,label=90:$x^3z$] {}
    -- ++(-30:1) node (302) [circle, fill=white,label=30:$x^3z^2$] {}
    -- ++(-90:1) node (202) [circle, fill=black,label=30:$x^2z^2$] {}
    -- ++(-30:1) node (203) [circle, fill=white,label=right:$x^2z^3$] {}
    -- ++(-90:1) node (103) [circle, fill=black,label=right:$xz^3$] {}
    -- ++(210:1) node (113) [circle, fill=white,label=-30:$xyz^3$] {}
    -- ++(-90:1) node (013) [circle, fill=black,label=-30:$yz^3$] {}
    -- ++(210:1) node (023) [circle, fill=white,label=-90:$y^2z^3$] {}
    -- ++(150:1) node (022) [circle, fill=black,label=-90:$y^2z^2$] {}
    -- ++(210:1) node (032) [circle, fill=white,label=-90:$y^3z^2$] {}
    -- ++(150:1) node (031) [circle, fill=black,label=210:$y^3z$] {}
    -- ++(90:1) node (131) [circle, fill=white,label=210:$xy^3z$] {}
    -- ++(150:1) node (130) [circle, fill=black,label=left:$xy^3$] {}
    -- ++(90:1) node (230) [circle, fill=white,label=left:$x^2y^3$] {}
    -- ++(30:1) node (220) [circle, fill=black,label=150:$x^2y^2$] {}
    -- ++(90:1) node (320) [circle, fill=white,label=150:$x^3y^2$] {}
    -- ++(30:1) node (310) [circle, fill=black,label=90:$x^3y$] {}
    -- (311);

  % fill in some edges
  \draw (220) -- (221);
  \draw (121) -- (131);
  \draw (122) -- (022);
  \draw (112) -- (113);
  \draw (212) -- (202);

  % get rid of shit
  \draw [thick, white] (310) -- (311);
  \draw [thick, white] (221) -- (211);
  \draw [thick, white] (121) -- (122);
  \draw [thick, white] (032) -- (022);
  \draw [white, fill=white] (-0.5,-4.5) rectangle (4,2);\\
};
\end{tikzpicture}
\caption{The graph $G$, with axis of symmetry shown, and the quotient $G'$ for $a=b=c=2$}
\label{fig:TSorbit}
\end{figure}

  Moreover, every vertex $[x^iy^jz^k]$ of $G'$ corresponds to two
  elements $x^iy^jz^k$ and $x^iz^jy^k$.  These vertices are in
  bijection with $x^iy^jz^k - x^iz^jy^k$ ($j > k$), the basis elements
  of $R^{\tau\kappa}$.  Again, we abuse notation and let $[i,j,k]$
  represent either a vertex of $G'$ or a basis element depending on
  context.  In~$G'$, every vertex orbit $[i, j, k]$ is connected to
  $[i+1, j, k]$, $[i, j +1, k]$ or $[i, j, k +1]$ (if those orbits
  exist), and similarly, $U_r$ sends any basis element $[i, j, k]$ ($j
  > k$) to a linear combination of basis elements $[i + 1, j, k] + [i,
  j + 1, k] + [i, j, k + 1]$ (if those elements exist.)  Notice if
  $k+1 = j$, then $[i,j,k+1] = 0$.

  Hence, $U'_r$, realized as a matrix indexed by the basis elements of
  rank $r$ and $r+1$, is the adjacency matrix for $G'$, and has
  entries in $\{0,1\}$.  Therefore the permanent of $U'_r$ counts the
  number of flip symmetric perfect matchings in $G$.

  To show permanent equals determinant, we show matchings in $G'$ have
  the same sign.  To get from one transpose complementary plane
  partition to another, we move a visible block (meaning three sides
  of the block are visible) to its transpose complementary location
  (which is empty by symmetry.)  This corresponds to hitting the
  perfect matching on $G'$ with a three cycle, so this preserves sign.
  All the transpose complementary plane partitions are connected in
  this way because we can transform them into the basic transpose
  complementary plane partition where the top layer of blocks is level
  (see the right plane partition in Figure~\ref{fig:blocks}.)  Hence,
  every matching on $G'$ has the same sign, making the permanent equal
  the determinant.

\begin{figure}
\begin{tikzpicture}[>= stealth]
\matrix{
\outline{1}{1}{1}{1};
\pp{{2,1}, {1}}{1};&
\draw (0,0) node {};
\draw [thick, ->] (0.25,0) node{} -- (1.25,0) node{};
\draw (1.5,0) node {};  &
\outline{1}{1}{1}{1};
\pp{{1,1},{1,1}}{1};\\
};
\end{tikzpicture}
\caption{Moving visible block to its transpose complementary location}
\label{fig:blocks}
\end{figure}
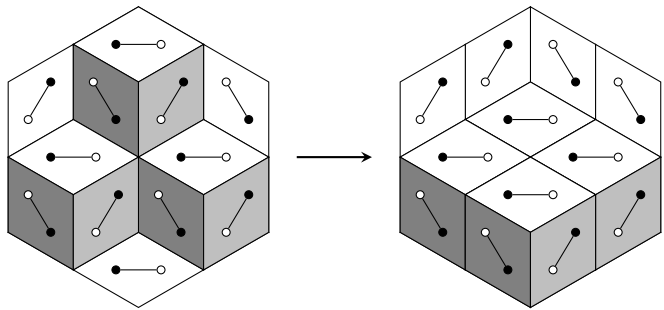

\end{proof}

%%%%%%%%%%%%%%%%%%%%%%%%%%%%%%%%%%%%%%%%%%%%%%
\subsection*{Proof of Theorem~\ref{t:cstcpp}}%%%%%%%%%%%%%
%%%%%%%%%%%%%%%%%%%%%%%%%%%%%%%%%%%%%%%%%%%%%%
Recall the statement of Theorem~\ref{t:cstcpp}:

\begin{fthm}
  Assuming $a=b=c$ are all even, let $C_2,C_3$ act on $R = \ZZ[x,y,z]/
  (x^{2a},y^{2a},z^{2a})$ as before. Then the map $U_m|_{R^{C_3} \cap
    R^{C_2,-}}$ restricted to the $m$-th homogeneous component of the
  intersection $R^{C_3} \cap R^{C_2,-}$ has $\det(U_m|_{R^{C_3} \cap
    R^{C_2,-}})$ equal, up to sign, to the number of cyclically
  symmetric transpose complementary plane partitions in an $a \times a
  \times a$ box.
\end{fthm}

\begin{proof}
This is a hybrid of the two previous proofs: the basis elements are
now $x^iy^jz^k - x^iz^ky^j + y^iz^jx^k - z^iy^jx^k + z^ix^jy^k -
y^ix^jz^k$, and now we move three visible blocks in the same $\rho$
orbit to their transpose complementary locations to get from one
CSTCPP to another.
\end{proof}

%%%%%%%%%%%%%%%%%%%%%%%%%%%%%%%%%%%%%%%%%%%
\section{Questions}
%%%%%%%%%%%%%%%%%%%%%%%%%%%%%%%%%%%%%%%%%%%
\begin{question}
  Is there some way to determine the Smith entries for the middle map
  $U_{\lfloor \frac{e-1}{2} \rfloor}$ for $R$, and is there a
  combinatorial explanation for the Smith entries in the context of
  plane partitions when $A = a + b, B = a + c$, and $C = b + c$?
\end{question}

\begin{question}
  Why does the extension of Theorem~\ref{t:snf} seem to work for four
  variables but not for five variables?
\end{question}

\begin{question}
  Is there some way to achieve the counts for the other symmetry
  classes of plane partitions as determinants of maps like $U_r$
  restricted to natural submodules of $R$?
\end{question}

% By glancing at the table in
% Figure~\ref{fig:PPtable}, one can see that it suffices to exhibit a
% degree-preserving action of $\kappa$ on $R$ in order to get the other
% six symmetry classes. Having such an action of $\kappa$ would
% immediately yield class~5. Furthermore, since $(\tau\kappa)\kappa =
% \tau$, we would get a degree-preserving action of $\tau$ on $R$, which
% would yield class~2. By intersecting various submodules, similar to
% Theorem~\ref{t:cstcpp}, we would be able to get the rest of the
% symmetry classes with these actions.  

%%%%%%%%%%%%%%%%%%%%%%%%%%%%%%%%%%%%%%%%%%%
\section*{Acknowledgements} %%%%%%%%%%%%%%%
%%%%%%%%%%%%%%%%%%%%%%%%%%%%%%%%%%%%%%%%%%%
This research was performed in a summer REU at the University of
Minnesota, mentored by Vic Reiner and Dennis Stanton, supported by NSF
grant DMS-1001933.  The authors thank Vic Reiner and Dennis Stanton
for their immense guidance, without which none of this research would
have been possible.

%%%%%%%%%%%%%%%%%%%%%%%%%%%%%%%%%%%%%%%%%%%
%%%%%%%%%%%%%%%%%%%%%%
%%%%%%%%%%%%%%%%%%%%%%%%%%%%%%%%%%%%%%%%%%%

%%%%%%%%%%%%%%%%%%%%%%%%%%%%%%%%%%%%%%%%%%%
\end{document}